\documentclass{article}

\usepackage[english]{babel}\usepackage[latin1]{inputenc}\usepackage[T1]{fontenc}
\usepackage{mathdots,graphicx,amsfonts,amssymb,amsmath,amsthm,mathrsfs,enumitem,color,mathtools}
\usepackage[top=2cm,bottom=3cm,left=2cm,right=2cm]{geometry}
\usepackage{sectsty}\sectionfont{\normalsize}\subsectionfont{\normalsize}

\usepackage{tikz}
\usetikzlibrary{matrix, calc}
\usetikzlibrary{fit}

\newtheorem{theorem}{Theorem}[section]

\newtheorem{proposition}{Proposition}[section]
\newtheorem{lemma}{Lemma}[section]
\newtheorem{corollary}{Corollary}[section]
\theoremstyle{definition}
\newtheorem{definition}{Definition}[section]
\newtheorem{remark}{Remark}[section]

\numberwithin{equation}{section}
\numberwithin{figure}{section}
\numberwithin{table}{section}

\renewcommand{\epsilon}{\varepsilon}
\renewcommand{\i}{{\rm i}}
\newcommand{\e}{{\rm e}}

\newcommand{\hh}{{\boldsymbol h}}
\newcommand{\ii}{{\boldsymbol i}}
\newcommand{\jj}{{\boldsymbol j}}
\newcommand{\kk}{{\boldsymbol k}}

\newcommand{\mm}{{\boldsymbol m}}
\newcommand{\nn}{{\boldsymbol n}}
\newcommand{\pp}{{\boldsymbol p}}

\newcommand{\uu}{{\boldsymbol u}}
\newcommand{\vv}{{\boldsymbol v}}
\newcommand{\btheta}{\boldsymbol\theta}

\newcommand{\bz}{\mathbf 0}
\newcommand{\bu}{\mathbf 1}
\newcommand{\ee}{{\mathbf e}}
\newcommand{\xx}{{\boldsymbol x}}

\renewcommand{\ss}{{\boldsymbol s}}
\renewcommand{\tt}{{\boldsymbol t}}

\newcommand{\rd}{\color{red}}

\newcommand{\bl}{\color{blue}}
\newcommand{\bk}{\color{black}}

\begin{document}

\title{Tensor product of GLT sequences}
\author{Carlo Garoni\\
\footnotesize Department of Mathematics, University of Rome Tor Vergata, Rome, Italy (garoni@mat.uniroma2.it)}
\date{}
\maketitle

\begin{abstract}
The theory of generalized locally Toeplitz (GLT) sequences is an apparatus for computing the spectral and singular value distribution of sequences of matrices that possess a (possibly hidden) Toeplitz-like structure. Sequences of this kind, which are known as GLT sequences, arise in several applications, including the discretization of differential and integral equations.
Associated with any GLT sequence is a special function called symbol. In this paper, we prove that, if $\{A_{n,1}\}_n,\ldots,\{A_{n,d}\}_n$ are GLT sequences with symbols $\kappa_1,\ldots,\kappa_d$, then their tensor (Kronecker) product $\{A_{n,1}\otimes\cdots\otimes A_{n,d}\}_n$ is a GLT sequence with symbol $\kappa_1\otimes\cdots\otimes\kappa_d$, up to suitable permutation matrices that only depend on the dimensions of the involved matrices $A_{n,1},\ldots,A_{n,d}$.
The permutation matrices in question are explicitly defined through a recursive formula that allows for their algorithmic computation.
Some applications of the presented result are discussed.

\smallskip


\smallskip

\end{abstract}

{\footnotesize\tableofcontents}

\section{Introduction}\label{intro}

Throughout this paper, a sequence of matrices is, by definition, a sequence of the form $\{A_n\}_n$, where $n$ varies in some infinite subset of $\mathbb N=\{1,2,3,\ldots\}$ and $A_n$ is a matrix of size $d_n\times e_n$ such that both $d_n$ and $e_n$ tend to $\infty$ as $n\to\infty$.
The theory of generalized locally Toeplitz (GLT) sequences is an apparatus for computing the spectral and singular value distribution of sequences of matrices that possess a (possibly hidden) Toeplitz-like structure. Sequences of this kind, which are known as GLT sequences, arise in several applications, including the discretization of differential and integral equations.
Nowadays, the theory of GLT sequences is a fairly extensive research area with numerous applications.
For readers who are new to the subject, we recommend the introduction \cite{GLT-intro} and the six-page conference paper \cite{aip}.
For a comprehensive exposition of the topic, we refer to the books \cite{GLTbookI,GLTbookII} and the book-like papers \cite{rg,bgR,bg,bgd}.
Recent noteworthy developments not published in the previous references include the identification between spaces of GLT sequences and function spaces \cite{b2017,krs2022}, the characterization of diagonal matrix sequences enjoying a spectral distribution in terms of GLT sequences \cite{b2019}, the spectral distribution result for non-Hermitian perturbations of Hermitian GLT sequences \cite{b2020}, the GLT classification of diagonal sampling matrix sequences obtained from quasi-uniform samples of almost everywhere continuous functions \cite{griglie_a.u.}, the derivation of a ``normal form'' for GLT sequences \cite{normal_form}, and the successful application to GLT sequences of both matrix-less spectral approximation methods \cite{bit2025} and block Jacobi/Gauss--Seidel preconditioning \cite{GLH}.
It is also worth noting that, as demonstrated by recent research, the spectral distribution of a GLT sequence has significant practical implications. For example, suppose that $\{A_n\}_n$ is a GLT sequence resulting from the discretization of a differential equation $\mathscr Au=f$ through a given numerical method. Then, the spectral distribution of $\{A_n\}_n$ can be used to measure the accuracy of the method in approximating the spectrum of the differential operator $\mathscr A$ \cite{DavideCalcolo2021}, to establish whether the method preserves the so-called average spectral gap \cite{DavideCalcolo2018}, or to formulate analytical predictions for the eigenvalues of both $A_n$ and $\mathscr A$ \cite{Tom-paper}. Moreover, the spectral distribution of $\{A_n\}_n$ can be exploited to design efficient iterative solvers for linear systems with matrix $A_n$ and to analyze/predict their performance; see \cite{BeKu,Kuij-SIREV} for accurate convergence estimates of Krylov methods based on the spectral distribution and \cite[p.~3]{GLTbookI} for more details on this subject.

Associated with any GLT sequence $\{A_n\}_n$ is a special function $\kappa$ called symbol.
The importance of the symbol $\kappa$ lies in the fact that it 
always describes the singular value distribution of $\{A_n\}_n$ and, in many cases of interest, it also describes the spectral distribution of $\{A_n\}_n$.
Given $d$ GLT sequences $\{A_{n,1}\}_n,\ldots,\{A_{n,d}\}_n$ with symbols $\kappa_1,\ldots,\kappa_d$, respectively, we can consider their tensor (Kronecker) product $\{A_{n,1}\otimes\cdots\otimes A_{n,d}\}_n$ and the tensor product of their symbols $\kappa_1\otimes\cdots\otimes\kappa_d$; see Section~\ref{sec:gptp} for the corresponding definitions. In the main result of this paper (Theorem~\ref{thm:GLTotimesGLT=GLT}), we prove that $\{A_{n,1}\otimes\cdots\otimes A_{n_d}\}_n$ is a GLT sequence with symbol $\kappa_1\otimes\cdots\otimes\kappa_d$, up to suitable permutation matrices that only depend on the dimensions of the involved matrices $A_{n,1},\ldots,A_{n,d}$. 
We also provide for the permutation matrices in question an explicit recursive definition that allows for their algorithmic computation; see Definition~\ref{Gamma(sigma)}.
It is worth noting that the proof of Theorem~\ref{thm:GLTotimesGLT=GLT} requires several preliminary results, some of which are significant enough to be considered as further main results of this paper in addition to Theorem~\ref{thm:GLTotimesGLT=GLT}; we expressly refer to Theorems~\ref{a.c.s.otimes}, \ref{thm:TotimesT=T}, \ref{thm:DotimesD=D} and Corollary~\ref{a.c.s.otimes-d}.

The paper is organized as follows.
In Section~\ref{sec:overview}, we give an overview of the theory of GLT sequences with a focus on the results that we need in this paper.
In Section~\ref{sec:gptp}, we collect some general properties of tensor products to be used in later sections.
In Section~\ref{sussec}, we revisit the notion of sparsely unbounded (s.u.)\ sequences of matrices in connection with tensor products.
In Sections~\ref{sec:tpacs}--\ref{sec:tpdsm}, we prove three tensor-product properties, one for the so-called approximating classes of sequences (a.c.s.), one for Toeplitz matrices and one for diagonal sampling matrices; these properties will be crucial in the proof of Theorem~\ref{thm:GLTotimesGLT=GLT}.
In Section~\ref{sec:tpGLT}, we state and prove Theorem~\ref{thm:GLTotimesGLT=GLT}.
In Section~\ref{sec:a}, we discuss some applications of Theorem~\ref{thm:GLTotimesGLT=GLT}.
In Section~\ref{sec:c}, we draw conclusions and suggest a possible future line of research.

\section{Overview of the theory of GLT sequences}\label{sec:overview}

In this section, we give an overview of the theory of GLT sequences.
For conciseness purposes, we only present the results that we need in this paper.
For a comprehensive exposition of the topic, see \cite{rg,bgR,bg,bgd,GLTbookI,GLTbookII}.
For an introduction to the subject, we recommend \cite{GLT-intro,aip}.

\subsection{Multi-index notation}\label{min}

A multi-index $\ii$ of length $d$, also called a $d$-index, is a vector in $\mathbb Z^d$; its length $d$ is denoted by $|\ii|$ and its components by $i_1,i_2,\ldots,i_d$.
$\bz$ and $\bu$ are the vectors of all zeros and all ones, respectively (their size will be clear from the context).
If $\mm$ is a $d$-index, we set $N(\mm)=m_1m_2\cdots m_d$.
A multi-index $\mm$ is said to be positive if its components are positive.
A sequence of multi-indices is a sequence of the form $\{\mm=\mm(n)\}_n$, where $\mm(n)$ is a multi-index for every $n$ and the length $|\mm(n)|$ is the same for all $n$.
If $\{\mm=\mm(n)\}_n$ is a sequence of positive multi-indices, we say that $\mm\to\infty$ as $n\to\infty$ if $\min(\mm)\to\infty$ as $n\to\infty$.
If $\hh,\kk$ are $d$-indices, an inequality such as $\hh\le\kk$ means that $h_i\le k_i$ for all $i=1,\ldots,d$.
If $\hh,\kk$ are $d$-indices such that $\hh\le\kk$, the $d$-index range $\{\hh,\ldots,\kk\}$ is the set $\{\ii\in\mathbb Z^d:\hh\le\ii\le\kk\}$. We assume for this set the standard lexicographic ordering:
\[ \Bigl[\ \ldots\ \bigl[\ [\ (i_1,\ldots,i_d)\ ]_{i_d=h_d,\ldots,k_d}\ \bigr]_{i_{d-1}=h_{d-1},\ldots,k_{d-1}}\ \ldots\ \Bigr]_{i_1=h_1,\ldots,k_1}. \]
For instance, in the case $d=2$ the ordering is
\begin{align*}
&(h_1,h_2),\,(h_1,h_2+1),\,\ldots,\,(h_1,k_2),\,(h_1+1,h_2),\,(h_1+1,h_2+1),\,\ldots,\,(h_1+1,k_2),\\
&\ldots\,\ldots\,\ldots,\,(k_1,h_2),\,(k_1,h_2+1),\,\ldots,\,(k_1,k_2).
\end{align*}
When a $d$-index $\ii$ varies in a $d$-index range $\{\hh,\ldots,\kk\}$ (this is often written as $\ii=\hh,\ldots,\kk$), it is understood that $\ii$ varies from $\hh$ to $\kk$ following the lexicographic ordering. For instance, if $\mm$ is a positive $d$-index and $\xx=[x_\ii]_{\ii=\bu}^\mm$, then $\xx$ is a vector of size $N(\mm)$ whose components $x_\ii$, $\ii=\bu,\ldots,\mm$, are ordered in accordance with the lexicographic ordering: the first component is $x_\bu=x_{(1,\ldots,1,1)}$, the second component is $x_{(1,\ldots,1,2)}$, and so on until the last component, which is $x_\mm=x_{(m_1,\ldots,m_d)}$. Similarly, if $X=[x_{\ii\jj}]_{\ii,\jj=\bu}^\mm$, then $X$ is an $N(\mm)\times N(\mm)$ matrix whose components are indexed by a pair of $d$-indices $\ii,\jj$, both varying in $\{\bu,\ldots,\mm\}$ following the lexicographic ordering.
If $\hh,\kk$ are $d$-indices such that $\hh\le\kk$, the notation $\sum_{\ii=\hh}^\kk$ indicates the summation over all $\ii$ in $\{\hh,\ldots,\kk\}$.
Operations involving $d$-indices (or general vectors with $d$ components) that have no meaning in the vector space $\mathbb R^d$ must always be interpreted in the componentwise sense. For instance, $\jj\hh=(j_1h_1,\ldots,j_dh_d)$, $\ii/\jj=(i_1/j_1,\ldots,i_d/j_d)$, etc.
For all $d$-indices $\ii,\jj$, we define $\delta_{\ii\jj}=1$ if $\ii=\jj$ and $\delta_{\ii\jj}=0$ otherwise.

\subsection{Singular value and spectral distribution of a sequence of matrices}\label{sec:distr}

Let $\mu_k$ be the Lebesgue measure in $\mathbb R^k$. Throughout this paper, all terminology from measure theory (such as ``measurable set'', ``measurable function'', ``a.e.'', etc.)\ always refers to the Lebesgue measure. A matrix-valued function $f:D\subseteq\mathbb R^k\to\mathbb C^{s\times t}$ is said to be measurable (respectively, continuous, continuous a.e., in $L^1(D)$, etc.)\ if its components $f_{ij}:D\to\mathbb C$, $i=1,\ldots,s$, $j=1,\ldots,t$, are measurable (respectively, continuous, continuous a.e., in $L^1(D)$, etc.). 
We denote by $C_c(\mathbb C)$ (respectively, $C_c(\mathbb R)$) the space of complex-valued functions defined on $\mathbb C$ (respectively, $\mathbb R$) with bounded support.
For every $x,y\in\mathbb R$, we define $x\wedge y=\min(x,y)$.
The singular values of a matrix $A\in\mathbb C^{m\times n}$ are denoted by $\sigma_i(A)$, $i=1,\ldots,m\wedge n$, and the eigenvalues of a matrix $A\in\mathbb C^{m\times m}$ are denoted by $\lambda_i(A)$, $i=1,\ldots,m$. 
Recall that a sequence of matrices is, by definition, a sequence of the form $\{A_n\}_n$, where $n$ varies in some infinite subset of $\mathbb N$ and $A_n$ is a matrix of size $d_n\times e_n$ such that both $d_n$ and $e_n$ tend to $\infty$ as $n\to\infty$.

\begin{definition}[\textbf{singular value and spectral distribution of a sequence of matrices}]\label{def-distribution}\hfill
\begin{itemize}[leftmargin=*,nolistsep]
	\item Let $\{A_n\}_n$ be a sequence of matrices with $A_n$ of size $d_n\times e_n$, and let $f:D\subset\mathbb R^k\to\mathbb C^{s\times t}$ be measurable with $0<\mu_k(D)<\infty$.
	We say that $\{A_n\}_n$ has a (asymptotic) singular value distribution described by $f$, and we write $\{A_n\}_n\sim_\sigma f$, if
	\begin{equation*}
	\lim_{n\rightarrow\infty}\frac1{d_n\wedge e_n}\sum_{i=1}^{d_n\wedge e_n}F(\sigma_i(A_n))=\frac1{\mu_k(D)}\int_D\frac{\sum_{i=1}^{s\wedge t}F(\sigma_i(f(\xx)))}{s\wedge t}{\rm d}\xx,\qquad\forall\,F\in C_c(\mathbb R).
	\end{equation*}
	\item Let $\{A_n\}_n$ be a sequence of square matrices with $A_n$ of size $d_n$, and let $f:D\subset\mathbb R^k\to\mathbb C^{s\times s}$ be measurable with $0<\mu_k(D)<\infty$.
	We say that $\{A_n\}_n$ has a (asymptotic) spectral (or eigenvalue) distribution described by $f$, and we write $\{A_n\}_n\sim_\lambda f$, if
	\begin{equation*}
	\lim_{n\rightarrow\infty}\frac1{d_n}\sum_{i=1}^{d_n}F(\lambda_i(A_n))=\frac1{\mu_k(D)}\int_D\frac{\sum_{i=1}^sF(\lambda_i(f(\xx)))}{s}{\rm d}\xx,\qquad\forall\,F\in C_c(\mathbb C).
	\end{equation*}
\end{itemize}
\end{definition}

We remark that the functions $\xx\mapsto\sum_{i=1}^{s\wedge t}F(\sigma_i(f(\xx)))$ and $\xx\mapsto\sum_{i=1}^{s}F(\lambda_i(f(\xx)))$ appearing in Definition~\ref{def-distribution} are well-defined and measurable by \cite[Lemma~2.1]{bg}. We refer the reader to \cite[Remarks~4.1--4.2]{blocking} for the informal meaning behind the singular value and spectral distribution of a sequence of matrices.

Throughout this paper, whenever we write a relation such as $\{A_n\}_n\sim_\sigma f$, without further specifications, it is understood that $\{A_n\}_n$ and $f$ are as in item~1 of Definition~\ref{def-distribution}, i.e., $\{A_n\}_n$ is a sequence of matrices and $f$ is a measurable function from a subset $D$ of some $\mathbb R^k$ with $0<\mu_k(D)<\infty$ to a matrix space $\mathbb C^{s\times t}$.
Similarly, whenever we write a relation such as $\{A_n\}_n\sim_\lambda f$, without further specifications, it is understood that $\{A_n\}_n$ and $f$ are as in item~2 of Definition~\ref{def-distribution}, i.e., $\{A_n\}_n$ is a sequence of square matrices and $f$ is a measurable function from a subset $D$ of some $\mathbb R^k$ with $0<\mu_k(D)<\infty$ to a matrix space $\mathbb C^{s\times s}$ consisting of square matrices.
Whenever we write $\{A_n\}_n\sim_{\sigma,\lambda}f$, we mean that both $\{A_n\}_n\sim_\sigma f$ and $\{A_n\}_n\sim_\lambda f$ hold.

\subsection{Special sequences of matrices}

We introduce in this section a few sequences of matrices that play an important role in the theory of GLT sequences.

\subsubsection{Zero-distributed sequences}
A zero-distributed sequence is a sequence of matrices $\{Z_n\}_n$ such that $\{Z_n\}_n\sim_\sigma0$, i.e.,
\[ \lim_{n\to\infty}\frac1{d_n\wedge e_n}\sum_{i=1}^{d_n\wedge e_n}F(\sigma_i(Z_n))=F(0),\qquad\forall\,F\in C_c(\mathbb R), \]
where $d_n\times e_n$ is the size of $Z_n$.

\subsubsection{Sequences of diagonal sampling matrices}
If $\nn\in\mathbb N^d$ and $a:[0,1]^d\to\mathbb C^{s\times t}$, the $\nn$th ($d$-level block) diagonal sampling matrix generated by $a$ is the block diagonal matrix of size $N(\nn)s\times N(\nn)t$ given by
\begin{equation*}
D_\nn(a)=\mathop{{\rm diag}}_{\ii=\bu,\ldots,\nn}a\Bigl(\frac\ii\nn\Bigr).
\end{equation*}
Any sequence of matrices of the form $\{D_\nn(a)\}_n$, with $a:[0,1]^d\to\mathbb C^{s\times t}$ and $\{\nn=\nn(n)\}_n\subseteq\mathbb N^d$ such that $\nn\to\infty$ as $n\to\infty$, is referred to as a sequence of ($d$-level block) diagonal sampling matrices generated by $a$.

\subsubsection{Toeplitz sequences}
If $f:[-\pi,\pi]^d\to\mathbb C^{s\times t}$ is a function in $L^1([-\pi,\pi]^d)$, its Fourier coefficients are denoted by $\{f_\kk\}_{\kk\in\mathbb Z^d}$ and are defined as follows:
\begin{equation}\label{fc-def}
f_\kk=\frac1{(2\pi)^d}\int_{[-\pi,\pi]^d}f(\btheta)\hspace{0.75pt}{\rm e}^{-{\rm i}\kk\cdot\btheta}{\rm d}\btheta\,\in\,\mathbb C^{s\times t},\qquad\kk\in\mathbb Z^d,
\end{equation}
where $\kk\cdot\btheta=k_1\theta_1+\ldots+k_d\theta_d$ is the scalar product and it is understood (here and throughout this paper) that the integral of a matrix-valued function is computed componentwise.
If $\nn\in\mathbb N^d$ and $f:[-\pi,\pi]^d\to\mathbb C^{s\times t}$ is a function in $L^1([-\pi,\pi]^d)$, the $\nn$th ($d$-level block) Toeplitz matrix generated by $f$ is the $N(\nn)s\times N(\nn)t$ matrix given by
\begin{equation}\label{Toep}
T_\nn(f)=[f_{\ii-\jj}]_{\ii,\jj=\bu}^\nn.
\end{equation}
Any sequence of matrices of the form $\{T_\nn(f)\}_n$, with $f:[-\pi,\pi]^d\to\mathbb C^{s\times t}$ in $L^1([-\pi,\pi]^d)$ and $\{\nn=\nn(n)\}_n\subseteq\mathbb N^d$ such that $\nn\to\infty$ as $n\to\infty$, is referred to as a ($d$-level block) Toeplitz sequence generated by $f$.

\subsection{Approximating classes of sequences}

The notion of approximating classes of sequences (a.c.s.)\ is the cornerstone of an asymptotic approximation theory for sequences of matrices that is fundamental for the theory of GLT sequences; see \cite[Chapter~5]{GLTbookI}. The formal definition of a.c.s.\ is reported in Definition~\ref{a.c.s.}. 
Throughout this paper, we denote by $\|\cdot\|$ the Euclidean norm ($2$-norm) of vectors and the associated induced (operator) norm for matrices.
Recall that, for every matrix $X$, we have $\|X\|=\sigma_{\max}(X)$, where $\sigma_{\max}(X)$ denotes the maximum singular value of $X$. 

\begin{definition}[\textbf{approximating class of sequences}]\label{a.c.s.}
Let $\{A_n\}_n$ be a sequence of matrices with $A_n$ of size $d_n\times e_n$, and let $\{\{B_{n,m}\}_n\}_m$ be a sequence of sequences of matrices with $B_{n,m}$ of size $d_n\times e_n$. We say that $\{\{B_{n,m}\}_n\}_m$ is an approximating class of sequences (a.c.s.)\ for $\{A_n\}_n$, and we write $\{B_{n,m}\}_n\xrightarrow{\rm a.c.s.}\{A_n\}_n$, if the following condition is met: for every $m$ there exists $n_m$ such that, for $n\ge n_m$,
\begin{equation*}
A_n=B_{n,m}+R_{n,m}+N_{n,m},\qquad{\rm rank}(R_{n,m})\le c(m)(d_n\wedge e_n),\qquad\|N_{n,m}\|\le\omega(m),
\end{equation*}
where $n_m,\,c(m),\,\omega(m)$ depend only on $m$ and $\lim_{m\to\infty}c(m)=\lim_{m\to\infty}\omega(m)=0$.
\end{definition}

\subsection{GLT sequences}


Definition~\ref{GLT_def} can be inferred from \cite[Section~5]{bgR}. It is an equivalent alternative to the usual (complicated) definition of GLT sequences.
Throughout this paper, we denote by $O_{s,t}$ the $s\times t$ zero matrix and by $I_s$ the $s\times s$ identity matrix. 

\begin{definition}[\textbf{generalized locally Toeplitz sequence}]\label{GLT_def}
Let $\{X_n\}_n$ be a sequence of matrices, with $X_n$ of size $N(\nn)s\times N(\nn)t$ for some fixed positive integers $s,t$ and some sequence of positive $d$-indices $\{\nn=\nn(n)\}_n$ tending to $\infty$, and let $\kappa:[0,1]^d\times[-\pi,\pi]^d\to\mathbb C^{s\times t}$ be measurable. We say that $\{X_n\}_n$ is a ($d$-level block) GLT sequence with symbol $\kappa$, and we write $\{X_n\}_n\sim_{\rm GLT}\kappa$, if there exist functions $a_{i,m}$, $f_{i,m}$, $i=1,\ldots,N_m$, such that:
\begin{itemize}[nolistsep,leftmargin=*]
	\item $a_{i,m}:[0,1]^d\to\mathbb C$ is continuous a.e.\ on $[0,1]^d$ and $f_{i,m}:[-\pi,\pi]^d\to\mathbb C^{s\times t}$ belongs to $L^1([-\pi,\pi]^d)$;
	\item $\kappa_m(\xx,\btheta)=\sum_{i=1}^{N_m}a_{i,m}(\xx)f_{i,m}(\btheta)\to\kappa(\xx,\btheta)$ a.e.\ on $[0,1]^d\times[-\pi,\pi]^d$; \vspace{2pt}
	\item $\{X_{n,m}\}_n=\bigl\{\sum_{i=1}^{N_m}D_\nn(a_{i,m}I_s)T_\nn(f_{i,m})\bigr\}_n\xrightarrow{\rm a.c.s.}\{X_n\}_n$.
\end{itemize}
\end{definition}

The properties of GLT sequences that we need in this paper are listed below. The corresponding proofs can be found in \cite{bgR,bg}.

\begin{enumerate}[nolistsep,leftmargin=37.5pt]
	\item[\textbf{GLT1.}] Let $\{X_n\}_n$ be a sequence of matrices, with $X_n$ of size $N(\nn)s\times N(\nn)t$ for some fixed positive integers $s,t$ and some sequence of positive $d$-indices $\{\nn=\nn(n)\}_n$ tending to $\infty$, and let $\kappa:[0,1]^d\times[-\pi,\pi]^d\to\mathbb C^{s\times t}$ be measurable.
	\begin{itemize}[nolistsep,leftmargin=*]
		\item If $\{X_n\}_n\sim_{\rm GLT}\kappa$, then $\{X_n\}_n\sim_\sigma\kappa$.
		\item If $\{X_n\}_n\sim_{\rm GLT}\kappa$ and the matrices $X_n$ are Hermitian, then $s=t$, $\kappa$ is Hermitian a.e.\ on $[0,1]^d\times[-\pi,\pi]^d$, and $\{X_n\}_n\sim_\lambda\kappa$.
	\end{itemize}
	\item[\textbf{GLT2.}] Let $s,t$ be positive integers and let $\{\nn=\nn(n)\}_n$ be a sequence of positive $d$-indices tending to $\infty$. Then,
	\begin{itemize}[nolistsep,leftmargin=*]
		\item $\{T_\nn(f)\}_n\sim_{\rm GLT}\kappa(\xx,\btheta)=f(\btheta)$ if $f:[-\pi,\pi]^d\to\mathbb C^{s\times t}$ belongs to $L^1([-\pi,\pi]^d)$;
		\item $\{D_\nn(a)\}_n\sim_{\rm GLT}\kappa(\xx,\btheta)=a(\xx)$ if $a:[0,1]^d\to\mathbb C^{s\times t}$ is continuous a.e.;
		\item for every sequence of matrices $\{Z_n\}_n$ with $Z_n$ of size $N(\nn)s\times N(\nn)t$, we have
		$\{Z_n\}_n\sim_{\rm GLT}\kappa(\xx,\btheta)=O_{s,t}$ if and only if $\{Z_n\}_n\sim_\sigma0$.
	\end{itemize}
	\item[\textbf{GLT3.}] Let $\{X_n\}_n$, $\{Y_n\}_n$ be sequences of matrices, with $X_n$ of size $N(\nn)s\times N(\nn)t$ and $Y_n$ of size $N(\nn)u\times N(\nn)v$ for some fixed positive integers $s,t,u,v$ and some sequence of positive $d$-indices $\{\nn=\nn(n)\}_n$ tending to $\infty$, and let $\kappa:[0,1]^d\times[-\pi,\pi]^d\to\mathbb C^{s\times t}$ and $\xi:[0,1]^d\times[-\pi,\pi]^d\to\mathbb C^{u\times v}$ be measurable. Suppose that $\{X_n\}_n\sim_{\rm GLT}\kappa$ and $\{Y_n\}_n\sim_{\rm GLT}\xi$. Then,
	\begin{itemize}[nolistsep,leftmargin=*]
		\item $\{\alpha X_n+\beta Y_n\}_n\sim_{\rm GLT}\alpha\kappa+\beta\xi$ for every $\alpha,\beta\in\mathbb C$ if $\kappa$ and $\xi$ are summable (i.e., $s=u$ and $t=v$);
		\item $\{X_nY_n\}_n\sim_{\rm GLT}\kappa\xi$ if $\kappa$ and $\xi$ are multipliable (i.e., $t=u$).
	\end{itemize}
	\item[\textbf{GLT4.}] Let $\{X_n\}_n$ be a sequence of matrices, with $X_n$ of size $N(\nn)s\times N(\nn)t$ for some fixed positive integers $s,t$ and some sequence of positive $d$-indices $\{\nn=\nn(n)\}_n$ tending to $\infty$, and let $\kappa:[0,1]^d\times[-\pi,\pi]^d\to\mathbb C^{s\times t}$ be measurable. Suppose that there exists a sequence of sequences of matrices $\{\{X_{n,m}\}_n\}_m$ with $X_{n,m}$ of size $N(\nn)s\times N(\nn)t$ such that
	\begin{itemize}[nolistsep,leftmargin=*]
		\item $\{X_{n,m}\}_n\sim_{\rm GLT}\kappa_m$ for some measurable $\kappa_m:[0,1]^d\times[-\pi,\pi]^d\to\mathbb C^{s\times t}$,
		\item $\kappa_m\to\kappa$ a.e.\ on $[0,1]^d\times[-\pi,\pi]^d$,
		\item $\{X_{n,m}\}_n\xrightarrow{\rm a.c.s.}\{X_n\}_n$.
	\end{itemize}
	Then, $\{X_n\}_n\sim_{\rm GLT}\kappa$.
\end{enumerate}

\section{General properties of tensor products}\label{sec:gptp}

We collect in this section a few general properties of tensor products that we need later on.
In particular, Section~\ref{tpm} explores the tensor (Kronecker) product of matrices, while Section~\ref{tpf} is devoted to the tensor product of functions.

\subsection{Tensor product of matrices}\label{tpm}
If $X\in\mathbb C^{m\times n}$ and $Y\in\mathbb C^{p\times q}$, the tensor (Kronecker) product of $X$ and $Y$ is the $mp\times nq$ matrix defined by
\[ X\otimes Y=[x_{ij}Y]_{i=1,\ldots,m}^{j=1,\ldots,n}=\begin{bmatrix}x_{11}Y & \cdots & x_{1n}Y\\ \vdots & & \vdots\\ x_{m1}Y & \cdots & x_{mn}Y\end{bmatrix}. \]
Here is a list of basic properties satisfied by tensor products. These properties will be used throughout this paper, either explicitly or implicitly.
Throughout this paper, if $\vv_1,\vv_2,\ldots,\vv_d$ are vectors, we denote by $(\vv_1,\vv_2,\ldots,\vv_d)$ the vector whose components are those of $\vv_1$ followed by those of $\vv_2$ followed by those of $\vv_3$ and so on until the last components, which are those of $\vv_d$.
\begin{enumerate}[leftmargin=22.5pt,nolistsep]
	\item[\textbf{P1.}] $[$Associativity$]$ $(X\otimes Y)\otimes Z=X\otimes(Y\otimes Z)$ for all matrices $X$, $Y$, $Z$. We can therefore omit parentheses in expressions like $X_1\otimes X_2\otimes\cdots\otimes X_d$.
	\item[\textbf{P2.}] $[$Bilinearity$]$ For each fixed matrix $X$, the map $Y\mapsto X\otimes Y$ is linear on $\mathbb C^{p\times q}$ for every $p,q\in\mathbb N$; for each fixed matrix $Y$, the map $X\mapsto X\otimes Y$ is linear on $\mathbb C^{m\times n}$ for every $m,n\in\mathbb N$.
	\item[\textbf{P3.}] 
	$(X\otimes Y)^T=X^T\otimes Y^T$ for all matrices $X$, $Y$.
	\item[\textbf{P4.}] $(X\otimes Y)(W\otimes Z)=(XW)\otimes(YZ)$ for all matrices $X$, $Y$, $W$, $Z$ such that $X$, $W$ are multipliable and $Y$, $Z$ are multipliable.
	\item[\textbf{P5.}] ${\rm rank}(X\otimes Y)={\rm rank}(X){\rm rank}(Y)$ and $\|X\otimes Y\|=\|X\|\,\|Y\|$ for all matrices $X$, $Y$. 
	\item[\textbf{P6.}] $[$Multi-index formula for tensor products$]$ If $\mm_1,\ldots,\mm_d,\nn_1,\ldots,\nn_d$ are positive multi-indices and $X_r\in\mathbb C^{N(\mm_r)\times N(\nn_r)}$ for $r=1,\ldots,d$, then
	\begin{equation}\label{mif}
	(X_1\otimes X_2\otimes\cdots\otimes X_d)_{\ii\jj}=(X_1)_{\ii_1\jj_1}(X_2)_{\ii_2\jj_2}\cdots(X_d)_{\ii_d\jj_d},\qquad\ii=\bu,\ldots,\mm,\qquad\jj=\bu,\ldots,\nn,
	\end{equation}
	where $\mm=(\mm_1,\ldots,\mm_d)$ and $\nn=(\nn_1,\ldots,\nn_d)$. In \eqref{mif}, it is understood that every $\ii\in\{\bu,\ldots,\mm\}$ is decomposed as $\ii=(\ii_1,\ldots,\ii_d)$ with $\ii_r\in\mathbb N^{|\mm_r|}$ for $r=1,\ldots,d$, and every $\jj\in\{\bu,\ldots,\nn\}$ is decomposed as $\jj=(\jj_1,\ldots,\jj_d)$ with $\jj_r\in\mathbb N^{|\nn_r|}$ for $r=1,\ldots,d$.
	As a particular case of \eqref{mif}, if $m_1,\ldots,m_d,n_1,\ldots,n_d$ are positive integers and $X_r\in\mathbb C^{m_r\times n_r}$ for $r=1,\ldots,d$, then 
	\[ (X_1\otimes X_2\otimes\cdots\otimes X_d)_{\ii\jj}=(X_1)_{i_1j_1}(X_2)_{i_2j_2}\cdots(X_d)_{i_dj_d},\qquad\ii=\bu,\ldots,\mm,\qquad\jj=\bu,\ldots,\nn, \]
	where $\mm=(m_1,\ldots,m_d)$ and $\nn=(n_1,\ldots,n_d)$.
\end{enumerate}
For properties {\bf P1}--{\bf P5}, see, e.g., \cite[Section~2.4.5]{GLTbookI} or \cite[Chapter~2]{Steeb}.
Property {\bf P6} follows from the combination of two definitions: the definition of tensor product and the definition of lexicographic ordering.
For example, in the case $d=2$, we can prove {\bf P6} as follows. Let $\mm_1,\mm_2,\nn_1,\nn_2$ be positive multi-indices, let $X_1\in\mathbb C^{N(\mm_1)\times N(\nn_1)}$, $X_2\in\mathbb C^{N(\mm_2)\times N(\nn_2)}$, and let $\mm=(\mm_1,\mm_2)$, $\nn=(\nn_1,\nn_2)$, as per {\bf P6} for $d=2$. Then,
\begin{align}
X_1\otimes X_2&=\left[\ (X_1)_{\ii_1\jj_1}X_2\ \right]_{\ii_1=\bu,\ldots,\mm_1}^{\jj_1=\bu,\ldots,\nn_1}\quad\mbox{\footnotesize(by definition of tensor product)}\notag\\
&=\left[\ \left[\ (X_1)_{\ii_1\jj_1}(X_2)_{\ii_2\jj_2}\ \right]_{\ii_2=\bu,\ldots,\mm_2}^{\jj_2=\bu,\ldots,\nn_2}\ \right]_{\ii_1=\bu,\ldots,\mm_1}^{\jj_1=\bu,\ldots,\nn_1}\notag\\
&=\left[\ (X_1)_{\ii_1\jj_1}(X_2)_{\ii_2\jj_2}\ \right]_{\ii=\bu,\ldots,\mm}^{\jj=\bu,\ldots,\nn}\quad\mbox{\footnotesize(by definition of lexicographic ordering),}\label{intui}
\end{align}
where in the last equality it is understood that $\ii=(\ii_1,\ii_2)$ and $\jj=(\jj_1,\jj_2)$ with $|\ii_1|=|\mm_1|$, $|\ii_2|=|\mm_2|$, $|\jj_1|=|\nn_1|$, $|\jj_2|=|\nn_2|$, 
as specified in {\bf P6}. From \eqref{intui} and the obvious equation $X_1\otimes X_2=\left[(X_1\otimes X_2)_{\ii\jj}\right]_{\ii=\bu,\ldots,\mm}^{\jj=\bu,\ldots,\nn}$, we infer that $(X_1\otimes X_2)_{\ii\jj}=(X_1)_{\ii_1\jj_1}(X_2)_{\ii_2\jj_2}$ for every $\ii=\bu,\ldots,\mm$ and $\jj=\bu,\ldots,\nn$. This is precisely \eqref{mif} for $d=2$, and the proof of {\bf P6} for $d=2$ is complete.
The proof of {\bf P6} for $d\ge3$ can be performed by induction on $d$, starting from $d=2$; the details are left to the reader.

In addition to {\bf P1}--{\bf P6}, a further property of tensor products, which is fundamental for our purposes, is highlighted in Lemma~\ref{GammaLd}.
In order to prove Lemma~\ref{GammaLd}, we need to introduce the permutation matrices $P_{n_1,n_2}$ and $\Gamma_{n_1,\ldots,n_d}(\sigma)$ (see Definitions~\ref{Gamma} and~\ref{Gamma(sigma)}) and to formulate a preliminary result (Lemma~\ref{GammaL}).
Throughout this paper, the (column) vectors of the canonical basis of $\mathbb C^n$ are denoted by $\ee_1^{(n)},\ldots,\ee_n^{(n)}$.
Moreover, a permutation $\zeta$ of the set $\{1,\ldots,n\}$ is denoted by its list of values $[\zeta(1),\ldots,\zeta(n)]$.
For example, $\zeta=[1,3,2]$ is the permutation of $\{1,2,3\}$ defined by $\zeta(1)=1$, $\zeta(2)=3$, $\zeta(3)=2$.

\begin{definition}\label{Gamma}
For all positive integers $n_1,n_2$, we define $P_{n_1,n_2}$ as the permutation matrix of size $n_1n_2$ associated with the permutation $\zeta$ of $\{1,2,\ldots,n_1n_2\}$ given by
\begin{align*}
\zeta&=[1,n_2+1,2n_2+1,\ldots,(n_1-1)n_2+1,\\
&\hphantom{=[}\;2,n_2+2,2n_2+2,\ldots,(n_1-1)n_2+2,\\
&\hphantom{=[}\;\ldots\,\ldots\,\ldots,\\
&\hphantom{=[}\;n_2,2n_2,3n_2\ldots,n_1n_2],
\end{align*}
i.e.,
\[ \zeta(i)=((i-1)\,\textup{mod}\,n_1)n_2+\left\lfloor\frac{i-1}{n_1}\right\rfloor+1,\qquad i=1,\ldots,n_1n_2. \]
In other words, $P_{n_1,n_2}$ is the $n_1n_2\times n_1n_2$ matrix whose rows are $(\ee_{\zeta(1)}^{(n_1n_2)})^T,\ldots,(\ee_{\zeta(n_1n_2)}^{(n_1n_2)})^T$ (in this order).
In formulas,
\begin{equation}\label{Gamman1n2}
P_{n_1,n_2}=\begin{bmatrix}
I_{n_1} \otimes (\ee_1^{(n_2)})^T \\[3pt]
I_{n_1} \otimes (\ee_2^{(n_2)})^T \\[3pt]
\vdots \\[3pt]
I_{n_1} \otimes (\ee_{n_2}^{(n_2)})^T
\end{bmatrix}=\sum_{i=1}^{n_2}\ee_i^{(n_2)}\otimes I_{n_1}\otimes(\ee_i^{(n_2)})^T.
\end{equation}
\end{definition}

Lemma~\ref{GammaL} is known in the literature; see, e.g., \cite[Theorem~2.8]{Steeb}.
For the reader's convenience, however, we include the details of the proof, which takes up only a few lines.

\begin{lemma}\label{GammaL}
For all matrices $X_1\in\mathbb C^{m_1\times n_1}$ and $X_2\in\mathbb C^{m_2\times n_2}$, we have
\begin{equation}\label{tensorP}
X_2\otimes X_1=P_{m_1,m_2}(X_1\otimes X_2)P_{n_1,n_2}^T.
\end{equation}
\end{lemma}
\begin{proof}
The result can be proved by a direct componentwise computation, i.e., by showing that the $(i,j)$ entry of the second matrix in \eqref{tensorP} is equal to the $(i,j)$ entry of the first matrix in \eqref{tensorP} for all indices $i=1,\ldots,m_1m_2$ and $j=1,\ldots,n_1n_2$.
A more elegant proof, based on formula \eqref{Gamman1n2} and properties {\bf P3}--{\bf P4}, is the following: for every $X_1\in\mathbb C^{m_1\times n_1}$ and every $X_2\in\mathbb C^{m_2\times n_2}$, we have
\begin{align*}
P_{m_1,m_2}(X_1\otimes X_2)P_{n_1,n_2}^T&=\Biggl(\sum_{i=1}^{m_2}\ee_i^{(m_2)}\otimes I_{m_1}\otimes(\ee_i^{(m_2)})^T\Biggr)(X_1\otimes X_2)\Biggl(\sum_{j=1}^{n_2}\ee_j^{(n_2)}\otimes I_{n_1}\otimes(\ee_j^{(n_2)})^T\Biggr)^T\\
&=\Biggl(\sum_{i=1}^{m_2}\ee_i^{(m_2)}\otimes I_{m_1}\otimes(\ee_i^{(m_2)})^T\Biggr)(X_1\otimes X_2)\Biggl(\sum_{j=1}^{n_2}(\ee_j^{(n_2)})^T\otimes I_{n_1}\otimes\ee_j^{(n_2)}\Biggr)\\
&=\sum_{i=1}^{m_2}\sum_{j=1}^{n_2}(\ee_i^{(m_2)}\otimes I_{m_1}\otimes(\ee_i^{(m_2)})^T)(1\otimes X_1\otimes X_2)((\ee_j^{(n_2)})^T\otimes I_{n_1}\otimes\ee_j^{(n_2)})\\
&=\sum_{i=1}^{m_2}\sum_{j=1}^{n_2}\ee_i^{(m_2)}(\ee_j^{(n_2)})^T\otimes X_1\otimes\underbrace{(\ee_i^{(m_2)})^TX_2\ee_j^{(n_2)}}_{(X_2)_{ij}}\\
&=\underbrace{\sum_{i=1}^{m_2}\sum_{j=1}^{n_2}(X_2)_{ij}\ee_i^{(m_2)}(\ee_j^{(n_2)})^T}_{X_2}{}\otimes X_1=X_2\otimes X_1. \qedhere
\end{align*}
\end{proof}

\begin{definition}\label{Gamma(sigma)}
For all positive integers $d,n_1,\ldots,n_d$ and all permutations $\sigma$ of $\{1,\ldots,d\}$, we recursively define $\Gamma_{n_1,\ldots,n_d}(\sigma)$ as the following permutation matrix of size $n_1\cdots n_d$.
\begin{itemize}[nolistsep,leftmargin=*]
	\item If $d=1$, then $\Gamma_{n_1}([1])=I_{n_1}$.
	\item If $d=2$, then $\Gamma_{n_1,n_2}([1,2])=I_{n_1n_2}$ and $\Gamma_{n_1,n_2}([2,1])=P_{n_1,n_2}$. 
	\item If $d\ge3$, then
	\begin{equation}\label{Gamma-rec}
	\Gamma_{n_1,\ldots,n_d}(\sigma)=(I_{n_{\sigma(1)}\cdots n_{\sigma(i-1)}}\otimes P_{n_{\sigma(i+1)}\cdots n_{\sigma(d)},n_d})(\Gamma_{n_1,\ldots,n_{d-1}}(\tau)\otimes I_{n_d}),
	\end{equation}
	where $i\in\{1,\ldots,d\}$ is the index such that $\sigma(i)=d$ and $\tau$ is the permutation of $\{1,\ldots,d-1\}$ defined by 
	\[ \tau=[\sigma(1),\ldots,\sigma(i-1),\sigma(i+1),\ldots,\sigma(d)]. \]
	In \eqref{Gamma-rec}, it is understood (as usual) that an empty product is equal to~1, so that
	\[ I_{n_{\sigma(1)}\cdots n_{\sigma(i-1)}}\otimes P_{n_{\sigma(i+1)}\cdots n_{\sigma(d)},n_d}=\left\{\begin{aligned}
	&I_1\otimes P_{n_{\sigma(2)}\cdots n_{\sigma(d)},n_d}=P_{n_1\cdots n_{d-1},n_d}, &\mbox{if $i=1$},\\
	&I_{n_{\sigma(1)}\cdots n_{\sigma(d-1)}}\otimes P_{1,n_d}=I_{n_{\sigma(1)}\cdots n_{\sigma(d-1)}}\otimes I_{n_d}=I_{n_1\cdots n_d}, &\mbox{if $i=d$}.
	\end{aligned}\right. \]
\end{itemize}
\end{definition}

It is left as an exercise for the reader to prove by induction on $d$ that, for all positive integers $d,n_1,\ldots,n_d$ and all permutations $\sigma$ of $\{1,\ldots,d\}$, $\Gamma_{n_1,\ldots,n_d}(\sigma)$ is a permutation matrix of size $n_1\cdots n_d$, as claimed in Definition~\ref{Gamma(sigma)}.

\begin{lemma}\label{GammaLd}
For all positive integers $d,m_1,\ldots,m_d,n_1,\ldots,n_d$, all permutations $\sigma$ of $\{1,\ldots,d\}$, and all matrices $X_1\in\mathbb C^{m_1\times n_1}$, $X_2\in\mathbb C^{m_2\times n_2}$, $\ldots$, $X_d\in\mathbb C^{m_d\times n_d}$, we have
\[ X_{\sigma(1)}\otimes\cdots\otimes X_{\sigma(d)}=\Gamma_{m_1,\ldots,m_d}(\sigma)(X_1\otimes\cdots\otimes X_d)(\Gamma_{n_1,\ldots,n_d}(\sigma))^T. \]
\end{lemma}
\begin{proof}
We prove the lemma by induction on $d$. For $d=1$, the result is obvious. 
For $d=2$, the result is clear when $\sigma$ is the identity $[1,2]$ and follows from Lemma~\ref{GammaL} when $\sigma=[2,1]$.
We fix $d\ge3$, we assume that the result is true for $d-1$, and we prove that it is true also for $d$.
Let $X_1\in\mathbb C^{m_1\times n_1}$, $X_2\in\mathbb C^{m_2\times n_2}$, $\ldots$, $X_d\in\mathbb C^{m_d\times n_d}$ and let $\sigma$ be a permutation of $\{1,\ldots,d\}$.
Denote by $i$ the index such that $\sigma(i)=d$, and let $\tau$ be the permutation of $\{1,\ldots,d-1\}$ defined by $\tau=[\sigma(1),\ldots,\sigma(i-1),\sigma(i+1),\ldots,\sigma(d)]$.
Then, we have
\begin{align*}
&X_{\sigma(1)}\otimes\cdots\otimes X_{\sigma(d)}=X_{\sigma(1)}\otimes\cdots\otimes X_{\sigma(i-1)}\otimes X_d\otimes X_{\sigma(i+1)}\otimes\cdots\otimes X_{\sigma(d)}\\
&=X_{\sigma(1)}\otimes\cdots\otimes X_{\sigma(i-1)}\otimes\left[P_{m_{\sigma(i+1)}\cdots m_{\sigma(d)},m_d}(X_{\sigma(i+1)}\otimes\cdots\otimes X_{\sigma(d)}\otimes X_d)P_{n_{\sigma(i+1)}\cdots n_{\sigma(d)},n_d}^T\right]\quad\mbox{\footnotesize(by Lemma~\ref{GammaL})}\\
&=(I_{m_{\sigma(1)}\cdots m_{\sigma(i-1)}}\otimes P_{m_{\sigma(i+1)}\cdots m_{\sigma(d)},m_d})(X_{\sigma(1)}\otimes\cdots\otimes X_{\sigma(i-1)}\otimes X_{\sigma(i+1)}\otimes\cdots\otimes X_{\sigma(d)}\otimes X_d)\,\cdot\\
&\qquad\cdot(I_{n_{\sigma(1)}\cdots n_{\sigma(i-1)}}\otimes P_{n_{\sigma(i+1)}\cdots n_{\sigma(d)},n_d}^T)\quad\mbox{\footnotesize(by {\bf P4})}\\
&=(I_{m_{\sigma(1)}\cdots m_{\sigma(i-1)}}\otimes P_{m_{\sigma(i+1)}\cdots m_{\sigma(d)},m_d})(X_{\tau(1)}\otimes\cdots\otimes X_{\tau(d-1)}\otimes X_d)(I_{n_{\sigma(1)}\cdots n_{\sigma(i-1)}}\otimes P_{n_{\sigma(i+1)}\cdots n_{\sigma(d)},n_d})^T\quad\mbox{\footnotesize(by {\bf P3})}\\
&=\left(I_{m_{\sigma(1)}\cdots m_{\sigma(i-1)}}\otimes P_{m_{\sigma(i+1)}\cdots m_{\sigma(d)},m_d}\right)\Bigl\{\left[\Gamma_{m_1,\ldots,m_{d-1}}(\tau)(X_1\otimes\cdots\otimes X_{d-1})(\Gamma_{n_1,\ldots,n_{d-1}}(\tau))^T\right]\otimes X_d\Bigr\}\,\cdot\\
&\qquad\cdot\left(I_{n_{\sigma(1)}\cdots n_{\sigma(i-1)}}\otimes P_{n_{\sigma(i+1)}\cdots n_{\sigma(d)},n_d}\right)^T\quad\mbox{\footnotesize(by induction hypothesis)}\\
&=\left(I_{m_{\sigma(1)}\cdots m_{\sigma(i-1)}}\otimes P_{m_{\sigma(i+1)}\cdots m_{\sigma(d)},m_d}\right)\left(\Gamma_{m_1,\ldots,m_{d-1}}(\tau)\otimes I_{m_d}\right)(X_1\otimes\cdots\otimes X_{d-1}\otimes X_d)\left(\Gamma_{n_1,\ldots,n_{d-1}}(\tau)\otimes I_{n_d}\right)^T\cdot\\
&\qquad\cdot\left(I_{n_{\sigma(1)}\cdots n_{\sigma(i-1)}}\otimes P_{n_{\sigma(i+1)}\cdots n_{\sigma(d)},n_d}\right)^T\quad\mbox{\footnotesize(by {\bf P3}--{\bf P4})}\\
&=\Gamma_{m_1,\ldots,m_d}(\sigma)(X_1\otimes\cdots\otimes X_d)(\Gamma_{n_1,\ldots,n_d}(\sigma))^T,
\end{align*}
where the last equality follows from Definition~\ref{GammaLd}.
\end{proof}


Lemma~\ref{uGamma} proves the uniqueness of the permutation matrix $\Gamma_{n_1,\ldots,n_d}(\sigma)$ appearing in Lemma~\ref{GammaLd}. More precisely, it shows that $\Gamma_{n_1,\ldots,n_d}(\sigma)$ is the unique permutation matrix (depending only on $n_1,\ldots,n_d,\sigma$) such that 
\begin{equation}\label{Ld-eq}
X_{\sigma(1)}\otimes\cdots\otimes X_{\sigma(d)}=\Gamma_{n_1,\ldots,n_d}(\sigma)(X_1\otimes\cdots\otimes X_d)(\Gamma_{n_1,\ldots,n_d}(\sigma))^T
\end{equation}
for all square matrices $X_1\in\mathbb C^{n_1\times n_1}$, $X_2\in\mathbb C^{n_2\times n_2}$, $\ldots$, $X_d\in\mathbb C^{n_d\times n_d}$.
Throughout this paper, the matrices of the canonical basis of $\mathbb C^{n\times n}$ are denoted by $E_{uv}^{(n)}$, $1\le u,v\le n$. In other words, for every $u,v=1,\ldots,n$, $E_{uv}^{(n)}$ is the matrix in $\mathbb C^{n\times n}$ having $1$ in position $(u,v)$ and $0$ elsewhere.

\begin{lemma}\label{uGamma}
Let $d,n_1,\ldots,n_d$ be positive integers, let $\sigma$ be a permutation of $\{1,\ldots,d\}$, and let $\Gamma$ be a permutation matrix of size $n_1\cdots n_d$
such that
\begin{equation}\label{Ld-eq'}
X_{\sigma(1)}\otimes\cdots\otimes X_{\sigma(d)}=\Gamma(X_1\otimes\cdots\otimes X_d)\Gamma^T
\end{equation}
for all square matrices $X_1\in\mathbb C^{n_1\times n_1}$, $X_2\in\mathbb C^{n_2\times n_2}$, $\ldots$, $X_d\in\mathbb C^{n_d\times n_d}$.
Then, $\Gamma=\Gamma_{n_1,\ldots,n_d}(\sigma)$.
\end{lemma}
\begin{proof}
The proof consists of three steps. In what follows, $d,n_1,\ldots,n_d,\sigma,\Gamma$ are as in the statement of the lemma and $\nn=(n_1,\ldots,n_d)$.

\medskip

\noindent
{\em Step 1.} Lemma~\ref{GammaLd} (applied with $m_1=n_1,\ldots,m_d=n_d$) yields \eqref{Ld-eq}. Equations~\eqref{Ld-eq}--\eqref{Ld-eq'} imply
\[ \Gamma_{n_1,\ldots,n_d}(\sigma)(X_1\otimes\cdots\otimes X_d)(\Gamma_{n_1,\ldots,n_d}(\sigma))^T=\Gamma(X_1\otimes\cdots\otimes X_d)\Gamma^T \]
for all square matrices $X_1\in\mathbb C^{n_1\times n_1}$, $\ldots$, $X_d\in\mathbb C^{n_d\times n_d}$.
Multiplying both sides by $\Gamma^T$ on the left and $\Gamma_{n_1,\ldots,n_d}(\sigma)$ on the right, we get the following equation, which holds for every $X_1\in\mathbb C^{n_1\times n_1}$, $\ldots$, $X_d\in\mathbb C^{n_d\times n_d}$:
\begin{equation}\label{p1-eq}
P(X_1\otimes\cdots\otimes X_d)=(X_1\otimes\cdots\otimes X_d)P,
\end{equation}
where $P=\Gamma^T\Gamma_{n_1,\ldots,n_d}(\sigma)$ is a permutation matrix as a product of two permutation matrices. The thesis of the lemma consists in proving that $P$ is the identity matrix.

\medskip

\noindent
{\em Step 2.} We show that
\begin{equation}\label{p2-eq}
PX=XP,
\end{equation}
for every $X\in\mathbb C^{N(\nn)\times N(\nn)}$. In what follows, for every $\uu,\vv=\bu,\ldots,\nn$, we denote by $E_{\uu\vv}^{(\nn)}$ the matrix in $\mathbb C^{N(\nn)\times N(\nn)}$ having $1$ in position $(\uu,\vv)$ and $0$ elsewhere. Note that the sets $\{E_{uv}^{(N(\nn))}:u,v=1,\ldots,N(\nn)\}$ and $\{E_{\uu\vv}^{(\nn)}:\uu,\vv=\bu,\ldots,\nn\}$ coincide, though the indexing of the elements is different.
\begin{itemize}[nolistsep,leftmargin=18pt]
	\item[(a)] Equation~\eqref{p2-eq} holds if $X$ is a matrix of the canonical basis $\{E_{\uu\vv}^{(\nn)}:\uu,\vv=\bu,\ldots,\nn\}$. This follows from \eqref{p1-eq}, because every matrix $E_{\uu\vv}^{(\nn)}$ can be written as a tensor product $X_1\otimes\cdots\otimes X_d$ for suitable $X_1\in\mathbb C^{n_1\times n_1}$, $\ldots$, $X_d\in\mathbb C^{n_d\times n_d}$. Indeed, we have
	\begin{equation}\label{ee}
	E_{\uu\vv}^{(\nn)}=E_{u_1v_1}^{(n_1)}\otimes\cdots\otimes E_{u_dv_d}^{(n_d)},\qquad\uu,\vv=\bu,\ldots,\nn.
	\end{equation}
	To prove \eqref{ee}, fix $\bu\le\uu,\vv\le\nn$. By {\bf P6}, for every $\ii,\jj=\bu,\ldots,\nn$, we have
	\[ (E_{u_1v_1}^{(n_1)}\otimes\cdots\otimes E_{u_dv_d}^{(n_d)})_{\ii\jj}=(E_{u_1v_1}^{(n_1)})_{i_1j_1}\cdots(E_{u_dv_d}^{(n_d)})_{i_dj_d}=\delta_{i_1u_1}\delta_{j_1v_1}\cdots\delta_{i_du_d}\delta_{j_dv_d}=\delta_{\ii\uu}\delta_{\jj\vv}=(E_{\uu\vv}^{(\nn)})_{\ii\jj}, \]
	hence $E_{u_1v_1}^{(n_1)}\otimes\cdots\otimes E_{u_dv_d}^{(n_d)}=E_{\uu\vv}^{(\nn)}$.
	\item[(b)] Since the matrices $E_{uv}^{(N(\nn))}$, $u,v=1,\ldots,N(\nn)$, are a basis of $\mathbb C^{N(\nn)\times N(\nn)}$, \eqref{p2-eq} holds for every $X\in\mathbb C^{N(\nn)\times N(\nn)}$. Indeed, every $X\in\mathbb C^{N(\nn)\times N(\nn)}$ can be written as a linear combination $X=\sum_{u,v=1}^{N(\nn)}x_{uv}E_{uv}^{(N(\nn))}$, hence
	\[ PX=P\sum_{u,v=1}^{N(\nn)}x_{uv}E_{uv}^{(N(\nn))}=\sum_{u,v=1}^{N(\nn)}x_{uv}PE_{uv}^{(N(\nn))}=\sum_{u,v=1}^{N(\nn)}x_{uv}E_{uv}^{(N(\nn))}P=\Biggl[\,\sum_{u,v=1}^{N(\nn)}x_{uv}E_{uv}^{(N(\nn))}\Biggr]P=XP, \]
	where the third equality is due to (a).
\end{itemize}

\medskip

\noindent
{\em Step 3.} Let $N=N(\nn)$.\,\footnote{\,The reason why we set $N=N(\nn)$ is to highlight that the argument of Step~3 actually holds for any positive integer $N$.} The last step of the proof consists in observing that the only permutation matrix $P\in\mathbb C^{N\times N}$ that satisfies \eqref{p2-eq} for every $X\in\mathbb C^{N\times N}$, i.e., the only matrix $P\in\mathbb C^{N\times N}$ that commutes with every other matrix $X\in\mathbb C^{N\times N}$, is the identity matrix $I_N$. This can be seen by choosing $X=E_{uv}^{(N)}$ in \eqref{p2-eq} for every $u,v=1,\ldots,N$. With this choice, we obtain $PE_{uv}^{(N)}=E_{uv}^{(N)}P$, i.e.,
\[ \underbrace{\left[\begin{array}{c|c|c}
& p_{1u} & \\
O \ & \vdots & \ O \\
& p_{Nu} &
\end{array}\right]}_{\substack{\textup{the $v$th column is the}\\\textup{only nonzero column}}}=\underbrace{\left[\begin{array}{ccc}
& O & \vphantom{\Big|}\\
\hline
p_{v1} & \cdots & p_{vN} \vphantom{\Big|}\\
\hline
& O & \vphantom{\Big|}
\end{array}\right]}_{\substack{\textup{the $u$th row is the}\\\textup{only nonzero row}}} \]
for every $u,v=1,\ldots,N$. This is equivalent to saying that
\[ p_{uu}=p_{vv},\qquad p_{1u}=\ldots=p_{u-1,u}=p_{u+1,u}=\ldots=p_{Nu}=0,\qquad p_{v1}=\ldots=p_{v,v-1}=p_{v,v+1}=\ldots=p_{vN}=0, \]
for every $u,v=1,\ldots,N$. Thus, we must have $P=\alpha I_N$ for some constant $\alpha$. Since $P$ is a permutation matrix, the constant $\alpha$ must be equal to~$1$.
\end{proof}

\subsection{Tensor product of functions}\label{tpf}

If $d,q_1,\ldots,q_d,s_1,\ldots,s_d,t_1,\ldots,t_d$ are positive integers and $g_i:E_i\subseteq\mathbb R^{q_i}\to\mathbb C^{s_i\times t_i}$ for every $i=1,\ldots,d$, the tensor-product function is defined as follows:
\begin{equation}\label{tpfun}
\begin{aligned}
&g_1\otimes\cdots\otimes g_d:E_1\times\cdots\times E_d\subseteq\mathbb R^{q_1+\ldots+q_d}\to\mathbb C^{(s_1\cdots s_d)\times(t_1\cdots t_d)},\\
&(g_1\otimes\cdots\otimes g_d)(\xx)=g_1(\xx_1)\otimes\cdots\otimes g_d(\xx_d),
\end{aligned}
\end{equation}
where $\xx=(\xx_1,\ldots,\xx_d)$ is decomposed in such a way that $\xx_i\in E_i$ for every $i=1,\ldots,d$.
Two basic results on the tensor product of functions that we need in this paper are collected in Lemma~\ref{tp-int} and Corollary~\ref{tpfc}.

\begin{lemma}\label{tp-int}
Let $d,q_1,\ldots,q_d,s_1,\ldots,s_d,t_1,\ldots,t_d$ be positive integers and let $g_i:E_i\subseteq\mathbb R^{q_i}\to\mathbb C^{s_i\times t_i}$ be a function in $L^1(E_i)$ for every $i=1,\ldots,d$. Then, the tensor-product function \eqref{tpfun} is in $L^1(E_1\times\cdots\times E_d)$ and we have
\begin{equation}\label{conto-g}
\int_{E_1\times\cdots\times E_d}(g_1\otimes\cdots\otimes g_d)=\biggl(\int_{E_1}g_1\biggr)\otimes\cdots\otimes\biggl(\int_{E_d}g_d\biggr).
\end{equation}
\end{lemma}
\begin{proof}
Let $\ss=(s_1,\ldots,s_d)$ and $\tt=(t_1,\ldots,t_d)$. For every $\ii=\bu,\ldots,\ss$ and $\jj=\bu,\ldots,\tt$, we have
\begin{align*}
\int_{E_1\times\cdots\times E_d}|(g_1\otimes\cdots\otimes g_d)_{\ii\jj}|&=\int_{E_1\times\cdots\times E_d}|((g_1\otimes\cdots\otimes g_d)(\xx))_{\ii\jj}|{\rm d}\xx\\
&=\int_{E_1\times\cdots\times E_d}|(g_1(\xx_1)\otimes\cdots\otimes g_d(\xx_d))_{\ii\jj}|{\rm d}\xx_1\cdots{\rm d}\xx_d\quad\mbox{\footnotesize(by \eqref{tpfun})}\\
&=\int_{E_1\times\cdots\times E_d}|(g_1(\xx_1))_{i_1j_1}\cdots(g_d(\xx_d))_{i_dj_d}|{\rm d}\xx_1\cdots{\rm d}\xx_d\quad\mbox{\footnotesize(by {\bf P6})}\\
&=\int_{E_1}|(g_1(\xx_1))_{i_1j_1}|{\rm d}\xx_1\cdots\int_{E_d}|(g_d(\xx_d))_{i_dj_d}|{\rm d}\xx_d\quad\mbox{\footnotesize(by Fubini's theorem).}
\end{align*}
The latter product is a finite number, because $g_i\in L^1(E_i)$ for every $i=1,\ldots,d$ by assumption. We therefore infer that $g_1\otimes\cdots\otimes g_d\in L^1(E_1\times\cdots\times E_d)$. Moreover, keeping in mind that every integral of a matrix-valued function is computed componentwise and repeating the previous equalities after removing everywhere the absolute values, we obtain, for every $\ii=\bu,\ldots,\ss$ and $\jj=\bu,\ldots,\tt$,
\begin{align*}
\biggl(\int_{E_1\times\cdots\times E_d}(g_1\otimes\cdots\otimes g_d)\biggr)_{\ii\jj}&=\int_{E_1\times\cdots\times E_d}(g_1\otimes\cdots\otimes g_d)_{\ii\jj}\\
&=\int_{E_1}(g_1(\xx_1))_{i_1j_1}{\rm d}\xx_1\cdots\int_{E_d}(g_d(\xx_d))_{i_dj_d}{\rm d}\xx_d\\
&=\biggl(\int_{E_1}g_1(\xx_1){\rm d}\xx_1\biggr)_{i_1j_1}\cdots\biggl(\int_{E_d}g_d(\xx_d){\rm d}\xx_d\biggr)_{i_dj_d}\\
&=\biggl(\biggl(\int_{E_1}g_1(\xx_1){\rm d}\xx_1\biggr)\otimes\cdots\otimes\biggl(\int_{E_d}g_d(\xx_d){\rm d}\xx_d\biggr)\biggr)_{\ii\jj}\quad\mbox{\footnotesize(by {\bf P6})}\\
&=\biggl(\biggl(\int_{E_1}g_1\biggr)\otimes\cdots\otimes\biggl(\int_{E_d}g_d\biggr)\biggr)_{\ii\jj}.
\end{align*}
This proves \eqref{conto-g}.
\end{proof}

\begin{corollary}\label{tpfc}
Let $d,q_1,\ldots,q_d,s_1,\ldots,s_d,t_1,\ldots,t_d$ be positive integers, let $f_i:[-\pi,\pi]^{q_i}\to\mathbb C^{s_i\times t_i}$ be a function in $L^1([-\pi,\pi]^{q_i})$ for every $i=1,\ldots,d$, and consider the tensor-product function
\begin{equation}\label{tpfun-f}
f_1\otimes\cdots\otimes f_d:[-\pi,\pi]^{q_1+\ldots+q_d}\to\mathbb C^{(s_1\cdots s_d)\times(t_1\cdots t_d)},\qquad(f_1\otimes\cdots\otimes f_d)(\btheta)=f_1(\btheta_1)\otimes\cdots\otimes f_d(\btheta_d),
\end{equation}
where $\btheta=(\btheta_1,\ldots,\btheta_d)$ is decomposed in such a way that $\btheta_i\in[-\pi,\pi]^{q_i}$ for every $i=1,\ldots,d$. Then, the Fourier coefficients of \eqref{tpfun-f} are given by
\begin{equation}\label{conto}
(f_1\otimes\cdots\otimes f_d)_\kk=(f_1)_{\kk_1}\otimes\cdots\otimes(f_d)_{\kk_d},\qquad\kk\in\mathbb Z^{q_1+\ldots+q_d},
\end{equation}
where $\kk=(\kk_1,\ldots,\kk_d)$ is decomposed in such a way that $\kk_i\in\mathbb Z^{q_i}$ for every $i=1,\ldots,d$.
\end{corollary}
\begin{proof}
For every $\kk=(\kk_1,\ldots,\kk_d)\in\mathbb Z^{q_1+\ldots+q_d}$ with $\kk_i\in\mathbb Z^{q_i}$ for $i=1,\ldots,d$, we have
\begin{align*}
&(f_1\otimes\cdots\otimes f_d)_\kk=\frac1{(2\pi)^{q_1+\ldots+q_d}}\int_{[-\pi,\pi]^{q_1+\ldots+q_d}}(f_1\otimes\cdots\otimes f_d)(\btheta)\hspace{0.75pt}\e^{-\i\kk\cdot\btheta}{\rm d}\btheta\quad\mbox{\footnotesize(by \eqref{fc-def})}\\
&=\frac1{(2\pi)^{q_1}}\cdots\frac1{(2\pi)^{q_d}}\int_{[-\pi,\pi]^{q_1}\times\cdots\times[-\pi,\pi]^{q_d}}\e^{-\i\kk_1\cdot\btheta_1}f_1(\btheta_1)\otimes\cdots\otimes\e^{-\i\kk_d\cdot\btheta_d}f_d(\btheta_d)\hspace{0.75pt}{\rm d}\btheta_1\cdots{\rm d}\btheta_d\quad\mbox{\footnotesize(by \eqref{tpfun-f} and {\bf P2})}\\
&=\biggl(\frac1{(2\pi)^{q_1}}\int_{[-\pi,\pi]^{q_1}}\e^{-\i\kk_1\cdot\btheta_1}f_1(\btheta_1)\hspace{0.75pt}{\rm d}\btheta_1\biggr)\otimes\cdots\otimes\biggl(\frac1{(2\pi)^{q_d}}\int_{[-\pi,\pi]^{q_d}}\e^{-\i\kk_d\cdot\btheta_d}f_d(\btheta_d)\hspace{0.75pt}{\rm d}\btheta_d\biggr)\quad\mbox{\footnotesize(by Lemma~\ref{tp-int})}\\
&=(f_1)_{\kk_1}\otimes\cdots\otimes(f_d)_{\kk_d}\quad\mbox{\footnotesize(by \eqref{fc-def}).}\qedhere
\end{align*}
\end{proof}

\section{S.u.~sequences of matrices and tensor products}\label{sussec}

The importance of sparsely unbounded (s.u.)\ sequences of matrices within the theory of GLT sequences is mainly due to \cite[Proposition~5.5]{GLTbookI}, which shows that the product of two a.c.s.\ related to two s.u.\ sequences $\{A_n\}_n$ and $\{A_n'\}_n$ is an a.c.s.\ for the product sequence $\{A_nA_n'\}_n$.
In this paper, the importance of s.u.\ sequences of matrices is due to the ``tensor-product version'' of \cite[Proposition~5.5]{GLTbookI}, which will be proved in Theorem~\ref{a.c.s.otimes}. In this section, we collect the results on s.u.\ sequences of matrices that are necessary for proving Theorem~\ref{a.c.s.otimes} and its corollary (Corollary~\ref{a.c.s.otimes-d}).

\begin{definition}[\textbf{sparsely unbounded sequence of matrices}]\label{s.u.def}
Let $\{A_n\}_n$ be a sequence of matrices with $A_n$ of size $d_n\times e_n$. We say that $\{A_n\}_n$ is sparsely unbounded (s.u.)\ if for every $M>0$ there exists $n_M$ such that, for $n\ge n_M$,
\[ \frac{\#\{i\in\{1,\ldots,d_n\wedge e_n\}:\,\sigma_i(A_n)>M\}}{d_n\wedge e_n}\le r(M), \]
where $\lim_{M\to\infty}r(M)=0$.
\end{definition}

Proposition~\ref{s.u.pro} provides equivalent characterizations of s.u.\ sequences of matrices.
For sequences of square matrices, Proposition~\ref{s.u.pro} is just \cite[Proposition~2.18]{bg}.
For general sequences of matrices, Proposition~\ref{s.u.pro} can be proved by adapting the argument used for proving \cite[Proposition~2.18]{bg}.
For completeness, however, we provide the details of the proof.
To this end, we need the following Theorem~\ref{m.s}, which is sometimes referred to as the minimax principle for singular values.
In what follows, the notation $V\subseteq_{\rm s}\mathbb C^n$ means that $V$ is a subspace (and not only a subset) of $\mathbb C^n$.

\begin{theorem}\label{m.s}
Let $X\in\mathbb C^{m\times n}$ and let $\sigma_1(X)\ge\ldots\ge\sigma_{m\wedge n}(X)$ be the singular values of $X$ sorted in descending order. Then,
\[ \sigma_i(X)=\max_{\substack{V\subseteq_{\rm s}\mathbb C^n\\ \dim V=i}}\ \min_{\substack{\xx\in V\\ \|\xx\|=1}}\|X\xx\|=\min_{\substack{V\subseteq_{\rm s}\mathbb C^n\\ \dim V=n-i+1}}\ \max_{\substack{\xx\in V\\ \|\xx\|=1}}\|X\xx\|,\qquad i=1,\ldots,m\wedge n. \]
\end{theorem}

Theorem~\ref{m.s} follows immediately from the minimax principle for eigenvalues \cite[Corollary~III.1.2]{Bhatia} applied to the Hermitian matrix $X^*X$, whose eigenvalues are equal to the squares of the singular values of $X$ plus further $n-(m\wedge n)$ zero eigenvalues to match the size of $X^*X$.

\begin{proposition}\label{s.u.pro}
Let $\{A_n\}_n$ be a sequence of matrices with $A_n$ of size $d_n\times e_n$. The following conditions are equivalent.
\begin{enumerate}[nolistsep,leftmargin=*]
	\item $\{A_n\}_n$ is s.u.
	\item We have
	\[ \lim_{M\to\infty}\limsup_{n\to\infty}\frac{\#\{i\in\{1,\ldots,d_n\wedge e_n\}:\,\sigma_i(A_n)>M\}}{d_n\wedge e_n}=0. \]
	\item For every $M>0$ there exists $n_M$ such that, for $n\ge n_M$,
	\begin{equation*}
	A_n=\hat A_{n,M}+\tilde A_{n,M},\qquad{\rm rank}(\hat A_{n,M})\le r(M)(d_n\wedge e_n),\qquad \|\tilde A_{n,M}\|\le M,
	\end{equation*}
	where $\lim_{M\to\infty}r(M)=0$.
\end{enumerate}
\end{proposition}
\begin{proof}
(1$\implies$2) Suppose that $\{A_n\}_n$ is s.u. Then, for every $M>0$ there exists $n_M$ such that, for $n\ge n_M$,
\[ \frac{\#\{i\in\{1,\ldots,d_n\wedge e_n\}:\,\sigma_i(A_n)>M\}}{d_n\wedge e_n}\le r(M), \]
where $\lim_{M\to\infty}r(M)=0$. Therefore, for every $M>0$, we have
\[ \limsup_{n\to\infty}\frac{\#\{i\in\{1,\ldots,d_n\wedge e_n\}:\,\sigma_i(A_n)>M\}}{d_n\wedge e_n}\le r(M). \]
As a consequence,
\[ \lim_{M\to\infty}\limsup_{n\to\infty}\frac{\#\{i\in\{1,\ldots,d_n\wedge e_n\}:\,\sigma_i(A_n)>M\}}{d_n\wedge e_n}=0. \]

(2$\implies$1) Suppose that condition~2 is met. For every $M>0$, define
\[\delta(M)=\limsup_{n\to\infty}\frac{\#\{i\in\{1,\ldots,d_n\wedge e_n\}:\,\sigma_i(A_n)>M\}}{d_n\wedge e_n}\in[0,1] \]
and note that (obviously)
\[ \limsup_{n\to\infty}\frac{\#\{i\in\{1,\ldots,d_n\wedge e_n\}:\,\sigma_i(A_n)>M\}}{d_n\wedge e_n}<\delta(M)+\frac1M. \]
Hence, by definition of $\limsup$, for every $M>0$, the sequence 
\[ \frac{\#\{i\in\{1,\ldots,d_n\wedge e_n\}:\,\sigma_i(A_n)>M\}}{d_n\wedge e_n} \] 
is eventually less than $r(M)=\delta(M)+1/M$, i.e., there exists $n_M$ such that, for $n\ge n_M$,
\[ \frac{\#\{i\in\{1,\ldots,d_n\wedge e_n\}:\,\sigma_i(A_n)>M\}}{d_n\wedge e_n}\le r(M). \]
Since $r(M)\to0$ as $M\to\infty$ by condition~2, condition~1 is proved.

(1$\implies$3) Suppose that $\{A_n\}_n$ is s.u.: for every $M>0$ there exists $n_M$ such that, for $n\ge n_M$,
\[ \frac{\#\{i\in\{1,\ldots,d_n\wedge e_n\}:\,\sigma_i(A_n)>M\}}{d_n\wedge e_n}\le r(M), \]
where $\lim_{M\to\infty}r(M)=0$. Let $A_n=U_n\Sigma_n V_n^*$ be a singular value decomposition of $A_n$. Let $\hat\Sigma_{n,M}$ be the matrix obtained from $\Sigma_n$ by setting to 0 all singular values of $A_n$ that are less than or equal to $M$, and let $\tilde\Sigma_{n,M}=\Sigma_n-\hat\Sigma_{n,M}$ be the matrix obtained from $\Sigma_n$ by setting to 0 all singular values of $A_n$ that exceed $M$. Then,
\[ A_n=U_n\Sigma_n V_n^*=U_n\hat\Sigma_{n,M}V_n^*+U_n\tilde\Sigma_{n,M}V_n^*=\hat A_{n,M}+\tilde A_{n,M}, \]
where $\hat A_{n,M}=U_n\hat\Sigma_{n,M}V_n^*$ and $\tilde A_{n,M}=U_n\tilde\Sigma_{n,M}V_n^*$ satisfy, for $n\ge n_M$,
\begin{align*}
{\rm rank}(\hat A_{n,M})&=\#\{i\in\{1,\ldots,d_n\wedge e_n\}:\,\sigma_i(A_n)>M\}\le r(M)(d_n\wedge e_n),\\
\|\tilde A_{n,M}\|&=\sigma_{\max}(\tilde A_{n,M})\le M.
\end{align*}

(3$\implies$1) Suppose that condition~3 holds. Then, for every $M>0$ there exists $n_M$ such that, for $n\ge n_M$,
\[ A_n=\hat A_{n,M}+\tilde A_{n,M},\qquad{\rm rank}(\hat A_{n,M})\le r(M)(d_n\wedge e_n),\qquad \|\tilde A_{n,M}\|\le M, \]
where $\lim_{M\to\infty}r(M)=0$. Assume that the singular values of matrices are sorted in descending order.
Then, by the minimax principle for singular values (Theorem~\ref{m.s}), for every $M>0$, every $n\ge n_M$, and every $i=1,\ldots,d_n\wedge e_n$, we have
\begin{align}
\sigma_i(A_n)&=\max_{\substack{V\subseteq_s\mathbb C^{e_n}\\\dim V=i}}\ \min_{\substack{\xx\in V\\\|\xx\|=1}}\|A_n\xx\|\le\max_{\substack{V\subseteq_s\mathbb C^{e_n}\\\dim V=i}}\ \min_{\substack{\xx\in V\\\|\xx\|=1}}\Bigl(\|\hat A_{n,M}\xx\|+\|\tilde A_{n,M}\xx\|\Bigr)\notag\\
&\le\max_{\substack{V\subseteq_s\mathbb C^{e_n}\\\dim V=i}}\ \min_{\substack{\xx\in V\\\|\xx\|=1}}\Bigl(\|\hat A_{n,M}\xx\|+\|\tilde A_{n,M}\|\Bigr)=\sigma_i(\hat A_{n,M})+\|\tilde A_{n,M}\|\le\sigma_i(\hat A_{n,M})+M.\label{sigma.0}
\end{align}
Since ${\rm rank}(\hat A_{n,M})\le r(M)(d_n\wedge e_n)$, the matrix $\hat A_{n,M}$ has at most $r(M)(d_n\wedge e_n)$ nonzero singular values. Therefore, by \eqref{sigma.0}, for every $M>0$ and every $n\ge n_M$, the matrix $A_n$ has at most $r(M)(d_n\wedge e_n)$ singular values greater than $M$, i.e., 
\[ \#\{i\in\{1,\ldots,d_n\wedge e_n\}:\,\sigma_i(A_n)>M\}\le r(M)(d_n\wedge e_n), \]
or, equivalently,
\[ \frac{\#\{i\in\{1,\ldots,d_n\wedge e_n\}:\,\sigma_i(A_n)>M\}}{d_n\wedge e_n}\le r(M). \]
This means that $\{A_n\}_n$ is s.u.
\end{proof}

In Proposition~\ref{d->su}, we show that any sequence of matrices enjoying a singular value distribution is s.u.
For sequences of square matrices, Proposition~\ref{d->su} is just \cite[Proposition~2.20]{bg}.
For general sequences of matrices, Proposition~\ref{d->su} can be proved by adapting the argument used for proving \cite[Proposition~2.20]{bg}.
For completeness, however, we provide the details of the proof. In what follows, $\chi_E$ denotes the characteristic (indicator) function of the set $E$.

\begin{proposition}\label{d->su}
If $\{A_n\}_n\sim_\sigma f$ then $\{A_n\}_n$ is s.u.
\end{proposition}
\begin{proof}
Suppose that $\{A_n\}_n\sim_\sigma f$. For every $n$, we denote by $d_n\times e_n$ the size of $A_n$. Moreover, we denote by $D\subset\mathbb R^k$ and $\mathbb C^{s\times t}$ the domain and codomain of the measurable function $f:D\to\mathbb C^{s\times t}$. 
Fix $M>0$ and take $F_M\in C_c(\mathbb R)$ such that $F_M=1$ over $[0,M/2]$, $F_M=0$ over $[M,\infty)$ and $0\le F_M\le 1$ over $\mathbb R$. Note that $F_M\le\chi_{[0,M]}$ over $[0,\infty)$. Then, for every $n$,
\begin{align*}
&\frac{\#\{i\in\{1,\ldots,d_n\wedge e_n\}:\,\sigma_i(A_n)>M\}}{d_n\wedge e_n}\\
&=1-\frac{\#\{i\in\{1,\ldots,d_n\wedge e_n\}:\,\sigma_i(A_n)\le M\}}{d_n\wedge e_n}=1-\frac1{d_n\wedge e_n}\sum_{i=1}^{d_n\wedge e_n}\chi_{[0,M]}(\sigma_i(A_n))\\
&\le1-\frac1{d_n\wedge e_n}\sum_{i=1}^{d_n\wedge e_n}F_M(\sigma_i(A_n))\xrightarrow{n\to\infty}1-\frac1{\mu_k(D)}\int_D\frac{\sum_{i=1}^{s\wedge t}F_M(\sigma_i(f(\xx)))}{s\wedge t}{\rm d}\xx.
\end{align*}
As a consequence,
\[ \limsup_{n\to\infty}\frac{\#\{i\in\{1,\ldots,d_n\wedge e_n\}:\,\sigma_i(A_n)>M\}}{d_n\wedge e_n}\le1-\frac1{\mu_k(D)}\int_D\frac{\sum_{i=1}^{s\wedge t}F_M(\sigma_i(f(\xx)))}{s\wedge t}{\rm d}\xx. \]
Since $\frac{1}{s\wedge t}\sum_{i=1}^{s\wedge t}F_M(\sigma_i(f(\xx)))\to1$ pointwise as $M\to\infty$ and $\bigl|\frac{1}{s\wedge t}\sum_{i=1}^{s\wedge t}F_M(\sigma_i(f(\xx)))\bigr|\le1$ for every $M>0$ and every $\xx\in D$, the dominated convergence theorem \cite[Theorem~1.34]{Rudinone} yields
\[ \lim_{M\to\infty}\int_D\frac{\sum_{i=1}^{s\wedge t}F_M(\sigma_{i}(f(\xx)))}{s\wedge t}{\rm d}\xx=\mu_k(D). \]
Thus,
\[ \lim_{M\to\infty}\limsup_{n\to\infty}\frac{\#\{i\in\{1,\ldots,d_n\wedge e_n\}:\,\sigma_i(A_n)>M\}}{d_n\wedge e_n}=0. \]
This means that condition~2 in Proposition~\ref{s.u.pro} is satisfied, i.e., $\{A_n\}_n$ is s.u.
\end{proof}

\begin{remark}\label{GLT->s.u.}
Any GLT sequence is s.u. Indeed, let $\{X_n\}_n$ be a GLT sequence with symbol $\kappa$ as per Definition~\ref{GLT_def}. Then, we have $\{X_n\}_n\sim_\sigma\kappa$ by {\bf GLT1}. Thus, $\{X_n\}_n$ is s.u.\ by Proposition~\ref{d->su}.
\end{remark}

The last result we need about s.u.\ sequences of matrices is Proposition~\ref{s.u.tensor}, which shows that the tensor product of two s.u.\ sequences of matrices is an s.u.\ sequence of matrices.

\begin{proposition}\label{s.u.tensor}
If $\{A_n\}_n$ and $\{A_n'\}_n$ are s.u.\ sequences of matrices, then $\{A_n\otimes A_n'\}_n$ is s.u.
\end{proposition}
\begin{proof}
Let $\{A_n\}_n$ and $\{A_n'\}_n$ be s.u.\ sequences of matrices.
For every $n$, we denote by $d_n\times e_n$ the size of $A_n$ and by $d_n'\times e_n'$ the size of $A_n'$.
By Proposition~\ref{s.u.pro}, for every $M>0$ there exists $n_M$ such that, for $n\ge n_M$,
\begin{align*}
A_n=\hat A_{n,M}+\tilde A_{n,M},\qquad {\rm rank}(\hat A_{n,M})\le r(M)(d_n\wedge e_n),\qquad \|\tilde A_{n,M}\|\le M,\\
A_n'=\hat A_{n,M}'+\tilde A_{n,M}',\qquad {\rm rank}(\hat A_{n,M}')\le r(M)(d_n'\wedge e_n'),\qquad \|\tilde A_{n,M}'\|\le M,
\end{align*}
where $\lim_{M\to\infty}r(M)=0$. Thus, for every $M>0$ and every $n\ge n_M$, we have
\begin{align*}
A_n\otimes A_n'&=\hat A_{n,M}\otimes A_n'+\tilde A_{n,M}\otimes\hat A_{n,M}'+\tilde A_{n,M}\otimes\tilde A_{n,M}'=\hat B_{n,M}+\tilde B_{n,M},
\end{align*}
where the matrices $\hat B_{n,M}=\hat A_{n,M}\otimes A'_n+\tilde A_{n,M}\otimes\hat A_{n,M}'$ and $\tilde B_{n,M}=\tilde A_{n,M}\otimes\tilde A_{n,M}'$
satisfy, by {\bf P5}, the following properties:
\begin{align*}
{\rm rank}(\hat B_{n,M})&\le2r(M)(d_n\wedge e_n)(d_n'\wedge e_n')\le2r(M)((d_nd_n')\wedge(e_ne_n')),\\
\|\tilde B_{n,M}\|&\le M^2.
\end{align*}
We conclude that $\{A_n\otimes A_n'\}_n$ is s.u.\ because condition~3 in Proposition~\ref{s.u.pro} is satisfied.
\end{proof}

\section{Tensor product of a.c.s.}\label{sec:tpacs}

Theorem~\ref{a.c.s.otimes} is the generalization of \cite[Proposition~5.5]{GLTbookI} to the case where the standard matrix product is replaced by the tensor product.
It shows that the tensor product of two a.c.s.\ related to two s.u.\ sequences of matrices $\{A_n\}_n$ and $\{A_n'\}_n$ is an a.c.s.\ for the tensor-product sequence $\{A_n\otimes A_n'\}_n$.

\begin{theorem}\label{a.c.s.otimes}
Let $\{A_n\}_n$, $\{A_n'\}_n$ be s.u.\ sequences of matrices 
and suppose that
\begin{itemize}[nolistsep,leftmargin=*]
	\item $\{B_{n,m}\}_n\xrightarrow{\rm a.c.s.}\{A_n\}_n$,
	\item $\{B_{n,m}'\}_n\xrightarrow{\rm a.c.s.}\{A_n'\}_n$.
\end{itemize}
Then $\{B_{n,m}\otimes B_{n,m}'\}_n\xrightarrow{\rm a.c.s.}\{A_n\otimes A_n'\}_n$.
\end{theorem}
\begin{proof}
For every $n$, we denote by $d_n\times e_n$ the size of $A_n$ and by $d_n'\times e_n'$ the size of $A_n'$.
By hypothesis, for every $m$ there exists $n_m$ such that, for $n\ge n_m$,
\begin{alignat*}{5}
A_n&=B_{n,m}+R_{n,m}+N_{n,m}, & \qquad {\rm rank}(R_{n,m})&\le c(m)(d_n\wedge e_n), & \qquad \|N_{n,m}\|&\le\omega(m),\\
A_n'&=B_{n,m}'+R_{n,m}'+N_{n,m}', & \qquad {\rm rank}(R_{n,m}')&\le c(m)(d_n'\wedge e_n'), & \qquad \|N_{n,m}'\|&\le\omega(m),
\end{alignat*}
where $\lim_{m\to\infty} c(m)=\lim_{m\to\infty}\omega(m)=0$. Hence, for every $m$ and every $n\ge n_m$,
\[ A_n\otimes A_n'=B_{n,m}\otimes B'_{n,m}+B_{n,m}\otimes R'_{n,m}+\boxed{B_{n,m}\otimes N'_{n,m}}+R_{n,m}\otimes A'_n+\boxed{N_{n,m}\otimes A'_n}\,. \]
By Proposition~\ref{s.u.pro}, since $\{A_n\}_n$ and $\{A'_n\}_n$ are s.u., for every $M>0$ there exists $n(M)$ such that, for $n\ge n(M)$,
\begin{alignat*}{5}
A_n&=\hat A_{n,M}+\tilde A_{n,M}, & \qquad {\rm rank}(\hat A_{n,M})&\le r(M)(d_n\wedge e_n), & \qquad \|\tilde A_{n,M}\|&\le M,\\
A_n'&=\hat A_{n,M}'+\tilde A_{n,M}', & \qquad {\rm rank}(\hat A_{n,M}')&\le r(M)(d_n'\wedge e_n'), & \qquad \|\tilde A_{n,M}'\|&\le M,
\end{alignat*}
where $\lim_{M\to\infty}r(M)=0$.
Setting $M_m=(\omega(m))^{-1/2}$, for every~$m$ and every $n\ge\max(n_m,n(M_m))$ we have
\begin{align*}
\boxed{B_{n,m}\otimes N'_{n,m}}+\boxed{N_{n,m}\otimes A'_n}&=(A_n-R_{n,m}-N_{n,m})\otimes N'_{n,m}+N_{n,m}\otimes(\hat A'_{n,M_m}+\tilde A'_{n,M_m})\\
&=(\hat A_{n,M_m}+\tilde A_{n,M_m}-R_{n,m}-N_{n,m})\otimes N'_{n,m}+N_{n,m}\otimes\hat A'_{n,M_m}+N_{n,m}\otimes\tilde A'_{n,M_m}\\
&=\bl\boxed{\bk\hat A_{n,M_m}\otimes N'_{n,m}}\bk+\rd\boxed{\bk\tilde A_{n,M_m}\otimes N'_{n,m}}\bk-\bl\boxed{\bk R_{n,m}\otimes N'_{n,m}}\bk-\rd\boxed{\bk N_{n,m}\otimes N'_{n,m}}\bk\\
&\qquad+\bl\boxed{\bk N_{n,m}\otimes\hat A'_{n,M_m}}\bk+\rd\boxed{\bk N_{n,m}\otimes\tilde A'_{n,M_m}}\bk,
\end{align*}
and so
\begin{align*}
A_n\otimes A'_n&=B_{n,m}\otimes B'_{n,m}+B_{n,m}\otimes R'_{n,m}+\boxed{B_{n,m}\otimes N'_{n,m}}+R_{n,m}\otimes A'_n+\boxed{N_{n,m}\otimes A'_n}\\
&=B_{n,m}\otimes B'_{n,m}+B_{n,m}\otimes R'_{n,m}+R_{n,m}\otimes A'_n+\bl\boxed{\bk\hat A_{n,M_m}\otimes N'_{n,m}}\bk+\rd\boxed{\bk\tilde A_{n,M_m}\otimes N'_{n,m}}\bk-\bl\boxed{\bk R_{n,m}\otimes N'_{n,m}}\bk\\
&\qquad-\rd\boxed{\bk N_{n,m}\otimes N'_{n,m}}\bk+\bl\boxed{\bk N_{n,m}\otimes\hat A'_{n,M_m}}\bk+\rd\boxed{\bk N_{n,m}\otimes\tilde A'_{n,M_m}}\bk,
\end{align*}
where, by {\bf P5},
\begin{align*}
&{\rm rank}\Bigl(B_{n,m}\otimes R'_{n,m}+R_{n,m}\otimes A'_n+\bl\boxed{\bk\hat A_{n,M_m}\otimes N'_{n,m}}\bk-\bl\boxed{\bk R_{n,m}\otimes N'_{n,m}}\bk+\bl\boxed{\bk N_{n,m}\otimes\hat A'_{n,M_m}}\hspace{0.75pt}\bk\Bigr)\\
&\le(3c(m)+2r(M_m))(d_n\wedge e_n)(d_n'\wedge e_n')\le(3c(m)+2r(M_m))((d_nd_n')\wedge(e_ne_n')),\\
&\Bigl\|\,\rd\boxed{\bk\tilde A_{n,M_m}\otimes N'_{n,m}}\bk-\rd\boxed{\bk N_{n,m}\otimes N'_{n,m}}\bk+\rd\boxed{\bk N_{n,m}\otimes\tilde A'_{n,M_m}}\,\bk\Bigr\|\\
&\le2(\omega(m))^{1/2}+(\omega(m))^2.
\end{align*}
We conclude that $\{B_{n,m}\otimes B'_{n,m}\}_n\xrightarrow{\rm a.c.s.}\{A_n\otimes A'_n\}_n$ by Definition~\ref{a.c.s.}.
\end{proof}

\begin{corollary}\label{a.c.s.otimes-d}
Let $d$ be a positive integer. For every $i=1,\ldots,d$, let $\{A_n^{(i)}\}_n$ be an s.u.\ sequence of matrices and suppose that $\{B_{n,m}^{(i)}\}_n\xrightarrow{\rm a.c.s.}\{A_n^{(i)}\}_n$. Then,
\[ \{B_{n,m}^{(1)}\otimes\cdots\otimes B_{n,m}^{(d)}\}_n\xrightarrow{\rm a.c.s.}\{A_n^{(1)}\otimes\cdots\otimes A_n^{(d)}\}_n. \]
\end{corollary}
\begin{proof}
The result follows from Theorem~\ref{a.c.s.otimes} and Proposition~\ref{s.u.tensor}.
\end{proof}

\section{Tensor product of Toeplitz matrices}\label{sec:tpTm}

The purpose of this section is to prove Theorem~\ref{thm:TotimesT=T}, which shows that the tensor product of Toeplitz matrices generated by matrix-valued functions is equal to the Toeplitz matrix generated by the tensor-product function, up to suitable permutation matrices that only depend on the dimensions of the involved Toeplitz matrices.
We remark that the result of Theorem~\ref{thm:TotimesT=T} is already known in the case of Toeplitz matrices generated by univariate scalar functions. Indeed, if $d,n_1,\ldots,n_d$ are positive integers and $f_1,\ldots,f_d:[-\pi,\pi]\to\mathbb C$ are in $L^1([-\pi,\pi])$, then
\begin{equation}\label{known}
T_{n_1}(f_1)\otimes\cdots\otimes T_{n_d}(f_d)=T_\nn(f_1\otimes\cdots\otimes f_d),
\end{equation}
where $\nn=(n_1,\ldots,n_d)$; see \cite[Lemma~3.3]{GLTbookII}. Even in the case of Toeplitz matrices generated by univariate (square) matrix-valued functions, the result of Theorem~\ref{thm:TotimesT=T} is known, though with a non-explicit definition of the involved permutation matrices; see \cite[Lemma~2.45]{bgd}. The novelty of Theorem~\ref{thm:TotimesT=T} consists in the fact that: (a)~the result is now proved for the case of Toeplitz matrices generated by arbitrary multivariate matrix-valued functions; (b)~the involved permutation matrices are clearly defined (and can therefore be explicitly computed).

Before being able to prove Theorem~\ref{thm:TotimesT=T}, some work is needed.
We first point out that any matrix of the form
\begin{equation}\label{mbtm}
[a_{\ii-\jj}]_{\ii,\jj=\bu}^\nn\in\mathbb C^{N(\nn)s\times N(\nn)t},
\end{equation}
where $\nn\in\mathbb N^d$ and $a_\kk\in\mathbb C^{s\times t}$ for every $\kk=-(\nn-\bu),\ldots,\nn-\bu$, is referred to as a ($d$-level block) Toeplitz matrix.
The Toeplitz matrix \eqref{Toep} is just the matrix \eqref{mbtm} in which the blocks $a_\kk$ coincide with the Fourier coefficients $f_\kk$.
For $n\in\mathbb N$ and $k\in\mathbb Z$, let $J_n^{(k)}=T_n(\e^{\i k\theta})$ be the $n\times n$ matrix whose $(i,j)$ entry equals~1 if $i-j=k$ and~0 otherwise, i.e.,
\begin{equation}\label{Jnk}
(J_n^{(k)})_{ij}=\delta_{i-j,k},\qquad i,j=1,\ldots,n,\qquad n\in\mathbb N,\qquad k\in\mathbb Z.
\end{equation}
For $\nn\in\mathbb N^d$ and $\kk\in\mathbb Z^d$, let
\begin{equation}\label{Jnnkk}
J_\nn^{(\kk)}=J_{n_1}^{(k_1)}\otimes\cdots\otimes J_{n_d}^{(k_d)}=T_{n_1}(\e^{\i k_1\theta_1})\otimes\cdots\otimes T_{n_d}(\e^{\i k_d\theta_d})=T_\nn(\e^{\i\kk\cdot\btheta}),
\end{equation}
where in the last equality we used \eqref{known}.
Lemma~\ref{TT-lemma} provides an alternative expression for the Toeplitz matrix \eqref{mbtm}.
We remark that, in the literature, the result of Lemma~\ref{TT-lemma} is often taken as the definition of \eqref{mbtm}; see, e.g., \cite[p.~148, eq.~(1)]{TilliL1}.

\begin{lemma}\label{TT-lemma}
The Toeplitz matrix \eqref{mbtm} admits the following expression:
\begin{equation}\label{ue}
\left[a_{\ii-\jj}\right]_{\ii,\jj=\bu}^\nn=\sum_{\kk=-(\nn-\bu)}^{\nn-\bu}J_\nn^{(\kk)}\otimes a_\kk=\sum_{\kk=-(\nn-\bu)}^{\nn-\bu}T_\nn(\e^{\i\kk\cdot\btheta})\otimes a_\kk.
\end{equation}
\end{lemma}
\begin{proof}
The second equality in \eqref{ue} is obvious from \eqref{Jnnkk}. The first equality in \eqref{ue} is proved ``blockwise'', by showing that the $s\times t$ block in position $(\ii,\jj)$ of the matrix in the right-hand side is equal to $a_{\ii-\jj}$ for every $\ii,\jj=\bu,\ldots,\nn$.
By \eqref{Jnk}--\eqref{Jnnkk} and {\bf P6}, for every $\ii,\jj=\bu,\ldots,\nn$ we have
\begin{align*}
(J_\nn^{(\kk)})_{\ii\jj}&=(J_{n_1}^{(k_1)}\otimes\cdots\otimes J_{n_d}^{(k_d)})_{\ii\jj}=(J_{n_1}^{(k_1)})_{i_1j_1}\cdots(J_{n_d}^{(k_d)})_{i_dj_d}=\delta_{i_1-j_1,k_1}\cdots\delta_{i_d-j_d,k_d}=\delta_{\ii-\jj,\kk}.
\end{align*}
Therefore,
\[ \Biggl(\,\sum_{\kk=-(\nn-\bu)}^{\nn-\bu}J_\nn^{(\kk)}\otimes a_\kk\Biggr)_{\ii\jj}=\sum_{\kk=-(\nn-\bu)}^{\nn-\bu}(J_\nn^{(\kk)}\otimes a_\kk)_{\ii\jj}=\sum_{\kk=-(\nn-\bu)}^{\nn-\bu}(J_\nn^{(\kk)})_{\ii\jj}a_\kk=\sum_{\kk=-(\nn-\bu)}^{\nn-\bu}\delta_{\ii-\jj,\kk}a_\kk=a_{\ii-\jj}, \]
and \eqref{ue} is proved.
\end{proof}

In what follows, if $d,p_1,\ldots,p_d,q_1,\ldots,q_d$ are positive integers, we denote by $\Pi_{p_1,\ldots,p_d}^{q_1,\ldots,q_d}$ the permutation matrix of size $p_1\cdots p_dq_1\cdots q_d$ given by
\begin{equation}\label{Pi-mu}
\Pi_{p_1,\ldots,p_d}^{q_1,\ldots,q_d}=\Gamma_{p_1,\ldots,p_d,q_1,\ldots,q_d}(\sigma),\qquad\sigma=[1,d+1,2,d+2,3,d+3,\ldots,d,2d];
\end{equation}
see Definition~\ref{Gamma(sigma)} for the definition of $\Gamma_{p_1,\ldots,p_d,q_1,\ldots,q_d}(\sigma)$.
Note that, for all positive integers $d,p_1,\ldots,p_d$, we have
\begin{equation}\label{pi-eq}
\Pi_{p_1,\ldots,p_d}^{1,\ldots,1}=I_{p_1\cdots p_d}.
\end{equation}
Equation~\eqref{pi-eq} follows from the definition of $\Pi_{p_1,\ldots,p_d}^{1,\ldots,1}$ in \eqref{Pi-mu} and from Lemma~\ref{uGamma} applied with $d,n_1,\ldots,n_d,\sigma,\Gamma$ replaced by, respectively, $2d,p_1,\ldots,p_d,1,\ldots,1,[1,d+1,2,d+2,3,d+3,\ldots,d,2d],I_{p_1\cdots p_d}$.

\begin{theorem}\label{thm:TotimesT=T}
Let $d,s_1,\ldots,s_d,t_1,\ldots,t_d$ be positive integers, let $\nn_1,\ldots,\nn_d$ be positive multi-indices, and let $f_i:[-\pi,\pi]^{|\nn_i|}\to\mathbb C^{s_i\times t_i}$ be in $L^1([-\pi,\pi]^{|\nn_i|})$ for every $i=1,\ldots,d$.
Then,
\begin{equation}\label{TotimesT=T}
T_{\nn_1}(f_1)\otimes\cdots\otimes T_{\nn_d}(f_d)=\Pi_{N(\nn_1),\ldots,N(\nn_d)}^{s_1,\ldots,s_d}T_\nn(f_1\otimes\cdots\otimes f_d)(\Pi_{N(\nn_1),\ldots,N(\nn_d)}^{t_1,\ldots,t_d})^T,
\end{equation}
where $\nn=(\nn_1,\ldots,\nn_d)$ and $\Pi_{N(\nn_1),\ldots,N(\nn_d)}^{s_1,\ldots,s_d}$, $\Pi_{N(\nn_1),\ldots,N(\nn_d)}^{t_1,\ldots,t_d}$ are permutation matrices defined by~\eqref{Pi-mu}. In particular, if $s_1=\ldots=s_d=t_1=\ldots=t_d=1$, i.e., the functions $f_1,\ldots,f_d$ are scalar, then
\begin{equation}\label{TotimesT=T-scalar}
T_{\nn_1}(f_1)\otimes\cdots\otimes T_{\nn_d}(f_d)=T_\nn(f_1\otimes\cdots\otimes f_d).
\end{equation}
\end{theorem}
\begin{proof}
Equation~\eqref{TotimesT=T-scalar} follows from \eqref{TotimesT=T} and the fact that, if $s_1=\ldots=s_d=t_1=\ldots=t_d=1$, then the permutation matrices $\Pi_{N(\nn_1),\ldots,N(\nn_d)}^{s_1,\ldots,s_d}$, $\Pi_{N(\nn_1),\ldots,N(\nn_d)}^{t_1,\ldots,t_d}$ appearing in \eqref{TotimesT=T} are identity matrices; see \eqref{pi-eq}. We prove \eqref{TotimesT=T}.

We first note that, for every $\kk_1\in\mathbb Z^{|\nn_1|},\ldots,\kk_d\in\mathbb Z^{|\nn_d|}$, we have
\begin{equation}\label{T(e)oT(e)}
T_{\nn_1}(\e^{\i\kk_1\cdot\btheta_1})\otimes\cdots\otimes T_{\nn_d}(\e^{\i\kk_d\cdot\btheta_d})=T_\nn(\e^{\i\kk\cdot\btheta}),
\end{equation}
where $\kk=(\kk_1,\ldots,\kk_d)$. Equation~\eqref{T(e)oT(e)} is obtained by applying \eqref{known} to each factor $T_{\nn_i}(\e^{\i\kk_i\cdot\btheta_i})$ so as to expand it into a tensor product of $|\nn_i|$ Toeplitz matrices generated by scalar univariate functions of the form $T_n(\e^{\i k\theta})$. Once this is done, the tensor product in the left-hand side of \eqref{T(e)oT(e)} has been expanded into a tensor product of $\sum_{i=1}^d|\nn_i|$ Toeplitz matrices of the form $T_n(\e^{\i k\theta})$, and it suffices to apply again \eqref{known} to get \eqref{T(e)oT(e)}.

Now, the proof of \eqref{TotimesT=T} is a matter of computations. We have
\begin{align*}
&T_{\nn_1}(f_1)\otimes\cdots\otimes T_{\nn_d}(f_d)\\
&=\Biggl(\,\sum_{\kk_1=-(\nn_1-\bu)}^{\nn_1-\bu}T_{\nn_1}(\e^{\i\kk_1\cdot\btheta_1})\otimes (f_1)_{\kk_1}\Biggr)\otimes\cdots\otimes\Biggl(\,\sum_{\kk_d=-(\nn_d-\bu)}^{\nn_d-\bu}T_{\nn_d}(\e^{\i\kk_d\cdot\btheta_d})\otimes(f_d)_{\kk_d}\Biggr)\quad\mbox{\footnotesize(by Lemma~\ref{TT-lemma})}\\
&=\sum_{\kk_1=-(\nn_1-\bu)}^{\nn_1-\bu}\cdots\sum_{\kk_d=-(\nn_d-\bu)}^{\nn_d-\bu}T_{\nn_1}(\e^{\i\kk_1\cdot\btheta_1})\otimes (f_1)_{\kk_1}\otimes\cdots\otimes T_{\nn_d}(\e^{\i\kk_d\cdot\btheta_d})\otimes(f_d)_{\kk_d}\quad\mbox{\footnotesize(by {\bf P2})}\\
&=\sum_{\kk=-(\nn-\bu)}^{\nn-\bu}\Pi_{N(\nn_1),\ldots,N(\nn_d)}^{s_1,\ldots,s_d}\Bigl[T_{\nn_1}(\e^{\i\kk_1\cdot\btheta_1})\otimes\cdots\otimes T_{\nn_d}(\e^{\i\kk_d\cdot\btheta_d})\otimes (f_1)_{\kk_1}\otimes\cdots\otimes(f_d)_{\kk_d}\Bigr](\Pi_{N(\nn_1),\ldots,N(\nn_d)}^{t_1,\ldots,t_d})^T\\
&\qquad\mbox{\footnotesize(by definition of $\Pi_{N(\nn_1),\ldots,N(\nn_d)}^{q_1,\ldots,q_d}$ and Lemma~\ref{GammaLd}; we have set $\kk=(\kk_1,\ldots,\kk_d)$)}\\
&=\Pi_{N(\nn_1),\ldots,N(\nn_d)}^{s_1,\ldots,s_d}\left[\sum_{\kk=-(\nn-\bu)}^{\nn-\bu}T_{\nn}(\e^{\i\kk\cdot\btheta})\otimes (f_1\otimes\cdots\otimes f_d)_\kk\right](\Pi_{N(\nn_1),\ldots,N(\nn_d)}^{t_1,\ldots,t_d})^T\quad\mbox{\footnotesize(by \eqref{T(e)oT(e)} and Corollary~\ref{tpfc})}\\
&=\Pi_{N(\nn_1),\ldots,N(\nn_d)}^{s_1,\ldots,s_d}T_\nn(f_1\otimes\cdots\otimes f_d)(\Pi_{N(\nn_1),\ldots,N(\nn_d)}^{t_1,\ldots,t_d})^T\quad\mbox{\footnotesize(by Lemma~\ref{TT-lemma}).}
\end{align*}
This completes the proof of \eqref{TotimesT=T}.
\end{proof}

We conclude this section by observing that \eqref{known} is a special case of \eqref{TotimesT=T-scalar} in which $f_1,\ldots,f_d$ are univariate, i.e., $\nn_i=n_i\in\mathbb N$ for every $i=1,\ldots,d$.

\section{Tensor product of diagonal sampling matrices}\label{sec:tpdsm}

The purpose of this section is to prove Theorem~\ref{thm:DotimesD=D}, which is the version of Theorem~\ref{thm:TotimesT=T} for diagonal sampling matrices to be used in the proof of our main result (Theorem~\ref{thm:GLTotimesGLT=GLT}).
We first note that, if $d,s$ are positive integers, $a:[0,1]^d\to\mathbb C$ and $\nn\in\mathbb N^d$, then
\begin{equation}\label{f1}
D_\nn(aI_s)=\mathop{\rm diag}_{\ii=\bu,\ldots,\nn}a\Bigl(\frac\ii\nn\Bigr)I_s=\biggl[\,\mathop{\rm diag}_{\ii=\bu,\ldots,\nn}a\Bigl(\frac\ii\nn\Bigr)\biggr]\otimes I_s=D_\nn(a)\otimes I_s.
\end{equation}

\begin{theorem}\label{thm:DotimesD=D}
Let $d,s_1,\ldots,s_d$ be positive integers, let $\nn_1,\ldots,\nn_d$ be positive multi-indices, and let $a_i:[0,1]^{|\nn_i|}\to\mathbb C$ for every $i=1,\ldots,d$. Then,
\begin{equation}\label{DotimesD=D}
D_{\nn_1}(a_1I_{s_1})\otimes\cdots\otimes D_{\nn_d}(a_dI_{s_d})=\Pi_{N(\nn_1),\ldots,N(\nn_d)}^{s_1,\ldots,s_d}D_\nn((a_1\otimes\cdots\otimes a_d)I_{s_1\cdots s_d})(\Pi_{N(\nn_1),\ldots,N(\nn_d)}^{s_1,\ldots,s_d})^T,
\end{equation}
where $\nn=(\nn_1,\ldots,\nn_d)$ and $\Pi_{N(\nn_1),\ldots,N(\nn_d)}^{s_1,\ldots,s_d}$ is a permutation matrix defined by~\eqref{Pi-mu}.
In particular, if $s_1=\ldots=s_d=1$, then
\begin{equation}\label{DotimesD=D-scalar}
D_{\nn_1}(a_1)\otimes\cdots\otimes D_{\nn_d}(a_d)=D_\nn(a_1\otimes\cdots\otimes a_d).
\end{equation}
\end{theorem}
\begin{proof}
Equation~\eqref{DotimesD=D-scalar} follows from \eqref{DotimesD=D} and the fact that, if $s_1=\ldots=s_d=1$, then the permutation matrix $\Pi_{N(\nn_1),\ldots,N(\nn_d)}^{s_1,\ldots,s_d}$ appearing in \eqref{DotimesD=D} is the identity matrix; see \eqref{pi-eq}. We prove \eqref{DotimesD=D}.

We first note that
\begin{equation}\label{known-micatanto}
D_{\nn_1}(a_1)\otimes\cdots\otimes D_{\nn_d}(a_d)=D_\nn(a_1\otimes\cdots\otimes a_d).
\end{equation}
Equation~\eqref{known-micatanto} can be proved by a direct componentwise computation, as follows: by {\bf P6}, for every $\ii,\jj=\bu,\ldots,\nn$, if we write $\ii=(\ii_1,\ldots,\ii_d)$ and $\jj=(\jj_1,\ldots,\jj_d)$ with $|\ii_i|=|\jj_i|=|\nn_i|$ for every $i=1,\ldots,d$, then we have
\begin{align*}
(D_{\nn_1}(a_1)\otimes\cdots\otimes D_{\nn_d}(a_d))_{\ii\jj}&=(D_{\nn_1}(a_1))_{\ii_1\jj_1}\cdots(D_{\nn_d}(a_d))_{\ii_d\jj_d}=\delta_{\ii_1\jj_1}a_1\Bigl(\frac{\ii_1}{\nn_1}\Bigr)\cdots \delta_{\ii_d\jj_d}a_d\Bigl(\frac{\ii_d}{\nn_d}\Bigr)\\
&=\delta_{\ii\jj}\hspace{0.75pt}(a_1\otimes\cdots\otimes a_d)\Bigl(\frac\ii\nn\Bigr)=(D_\nn(a_1\otimes\cdots\otimes a_d))_{\ii\jj}.
\end{align*}
This completes the proof of \eqref{known-micatanto}.

Now, the proof of \eqref{DotimesD=D} is a matter of computations. We have
\begin{align*}
D_{\nn_1}(a_1I_{s_1})\otimes\cdots\otimes D_{\nn_d}(a_dI_{s_d})&=D_{\nn_1}(a_1)\otimes I_{s_1}\otimes\cdots\otimes D_{\nn_d}(a_d)\otimes I_{s_d}\quad\mbox{\footnotesize(by \eqref{f1})}\\
&=\Pi_{N(\nn_1),\ldots,N(\nn_d)}^{s_1,\ldots,s_d}\Bigl[D_{\nn_1}(a_1)\otimes\cdots\otimes D_{\nn_d}(a_d)\otimes I_{s_1}\otimes\cdots\otimes I_{s_d}\Bigr](\Pi_{N(\nn_1),\ldots,N(\nn_d)}^{s_1,\ldots,s_d})^T\\
&\qquad\mbox{\footnotesize(by definition of $\Pi_{N(\nn_1),\ldots,N(\nn_d)}^{s_1,\ldots,s_d}$ and Lemma~\ref{GammaLd})}\\
&=\Pi_{N(\nn_1),\ldots,N(\nn_d)}^{s_1,\ldots,s_d}\Bigl[D_{\nn}(a_1\otimes\cdots\otimes a_d)\otimes I_{s_1\cdots s_d}\Bigr](\Pi_{N(\nn_1),\ldots,N(\nn_d)}^{s_1,\ldots,s_d})^T\\
&\qquad\mbox{\footnotesize(by \eqref{known-micatanto} and the equality $I_{s_1}\otimes\cdots\otimes I_{s_d}=I_{s_1\cdots s_d}$)}\\
&=\Pi_{N(\nn_1),\ldots,N(\nn_d)}^{s_1,\ldots,s_d}\Bigl[D_{\nn}((a_1\otimes\cdots\otimes a_d)I_{s_1\cdots s_d})\Bigr](\Pi_{N(\nn_1),\ldots,N(\nn_d)}^{s_1,\ldots,s_d})^T\quad\mbox{\footnotesize(by \eqref{f1}).}
\end{align*}
This completes the proof of \eqref{DotimesD=D}.
\end{proof}

\begin{remark}\label{cancel_out}
The permutation matrices $\Pi_{N(\nn_1),\ldots,N(\nn_d)}^{s_1,\ldots,s_d}$ appearing in Theorems~\ref{thm:TotimesT=T} and~\ref{thm:DotimesD=D} coincide.
This fact will be crucial in the proof of Theorem~\ref{thm:GLTotimesGLT=GLT}.
\end{remark}

\section{Tensor product of GLT sequences}\label{sec:tpGLT}

We are now ready to prove the main result of this paper, i.e., Theorem~\ref{thm:GLTotimesGLT=GLT}. It shows that, if $\{A_{n,1}\}_n,\ldots,\{A_{n,d}\}_n$ are GLT sequences with symbols $\kappa_1,\ldots,\kappa_d$, then $\{A_{n,1}\otimes\cdots\otimes A_{n,d}\}_n$ is a GLT sequence with symbol $\kappa_1\otimes\cdots\otimes\kappa_d$, up to suitable permutation matrices that only depend on the dimensions of the involved matrices $A_{n,1},\ldots,A_{n,d}$.
The permutation matrices in question, which are the same as in Theorem~\ref{thm:TotimesT=T}, are explicitly defined by \eqref{Pi-mu} and can be computed through the recursive formula in Definition~\ref{Gamma(sigma)}.
In practice, Theorem~\ref{thm:GLTotimesGLT=GLT} carries over to GLT sequences the results proved in Theorems~\ref{thm:TotimesT=T} and~\ref{thm:DotimesD=D} for Toeplitz and diagonal sampling matrices. The key ingredients for this transfer process are Definition~\ref{GLT_def}, property {\bf GLT4} and Corollary~\ref{a.c.s.otimes-d}.
Before stating Theorem~\ref{thm:GLTotimesGLT=GLT}, we recall that, by definition, all multi-indices belonging to a given sequence of multi-indices $\{\mm=\mm(n)\}_n$ have the same length; see Section~\ref{min}.

\begin{theorem}\label{thm:GLTotimesGLT=GLT}
Let $d,s_1,\ldots,s_d,t_1,\ldots,t_d$ be positive integers, let $\{\nn_1=\nn_1(n)\}_n,\ldots,\{\nn_d=\nn_d(n)\}_n$ be sequences of positive multi-indices tending to $\infty$, and, for every $i=1,\ldots,d$, let $\{A_{n,i}\}_n\sim_{\rm GLT}\kappa_i$, where $A_{n,i}\in\mathbb C^{N(\nn_i)s_i\times N(\nn_i)t_i}$, $\kappa_i:[0,1]^{|\nn_i|}\times[-\pi,\pi]^{|\nn_i|}\to\mathbb C^{s_i\times t_i}$ is measurable, and $|\nn_i|$ denotes the common length of all multi-indices belonging to the $i$th sequence $\{\nn_i=\nn_i(n)\}_n$. Then,
\begin{equation}\label{GLTotimesGLT=GLT}
\{(\Pi_{N(\nn_1),\ldots,N(\nn_d)}^{s_1,\ldots,s_d})^T(A_{n,1}\otimes\cdots\otimes A_{n,d})\Pi_{N(\nn_1),\ldots,N(\nn_d)}^{t_1,\ldots,t_d}\}_n\sim_{\rm GLT}\kappa(\xx,\btheta)=\kappa_1(\xx_1,\btheta_1)\otimes\cdots\otimes\kappa_d(\xx_d,\btheta_d),
\end{equation}
where $\Pi_{N(\nn_1),\ldots,N(\nn_d)}^{s_1,\ldots,s_d}$, $\Pi_{N(\nn_1),\ldots,N(\nn_d)}^{t_1,\ldots,t_d}$ are permutation matrices defined by~\eqref{Pi-mu}. In \eqref{GLTotimesGLT=GLT}, it is understood that every $\xx=(\xx_1,\ldots,\xx_d)\in[0,1]^{|\nn_1|+\ldots+|\nn_d|}$ is decomposed in such a way that $\xx_i\in[0,1]^{|\nn_i|}$ for each $i=1,\ldots,d$, and every $\btheta=(\btheta_1,\ldots,\btheta_d)\in[-\pi,\pi]^{|\nn_1|+\ldots+|\nn_d|}$ is decomposed in such a way that $\btheta_i\in[-\pi,\pi]^{|\nn_i|}$ for each $i=1,\ldots,d$.
In particular, if $s_1=\ldots=s_d=t_1=\ldots=t_d=1$, i.e., the symbols $\kappa_1,\ldots,\kappa_d$ are scalar, then
\begin{equation}\label{GLTotimesGLT=GLT-scalar}
\{A_{n,1}\otimes\cdots\otimes A_{n,d}\}_n\sim_{\rm GLT}\kappa(\xx,\btheta)=\kappa_1(\xx_1,\btheta_1)\otimes\cdots\otimes\kappa_d(\xx_d,\btheta_d).
\end{equation}
\end{theorem}
\begin{proof}
Equation~\eqref{GLTotimesGLT=GLT-scalar} follows from \eqref{GLTotimesGLT=GLT} and the fact that, if $s_1=\ldots=s_d=t_1=\ldots=t_d=1$, then the permutation matrices $\Pi_{N(\nn_1),\ldots,N(\nn_d)}^{s_1,\ldots,s_d}$, $\Pi_{N(\nn_1),\ldots,N(\nn_d)}^{t_1,\ldots,t_d}$ appearing in \eqref{GLTotimesGLT=GLT} are identity matrices; see \eqref{pi-eq}. We prove \eqref{GLTotimesGLT=GLT}.

Set $\ss=(s_1,\ldots,s_d)$, $\tt=(t_1,\ldots,t_d)$, $\nn=(\nn_1,\ldots,\nn_d)$, and denote by $|\nn|=|\nn_1|+\ldots+|\nn_d|$ the common length of all multi-indices belonging to the sequence $\{\nn=\nn(n)\}_n$. Then, for every $n$ and every $(\xx,\btheta)\in[0,1]^{|\nn|}\times[-\pi,\pi]^{|\nn|}$, we have
\begin{align*}
(\Pi_{N(\nn_1),\ldots,N(\nn_d)}^{s_1,\ldots,s_d})^T(A_{n,1}\otimes\cdots\otimes A_{n,d})\Pi_{N(\nn_1),\ldots,N(\nn_d)}^{t_1,\ldots,t_d}&\in\mathbb C^{N(\nn)N(\ss)\times N(\nn)N(\tt)},\\
\kappa(\xx,\btheta)&\in\mathbb C^{N(\ss)\times N(\tt)}.
\end{align*}
Moreover, $\nn\to\infty$ as $n\to\infty$ and $\kappa:[0,1]^{|\nn|}\times[-\pi,\pi]^{|\nn|}\to\mathbb C^{N(\ss)\times N(\tt)}$ is measurable. Thus, the GLT relation \eqref{GLTotimesGLT=GLT} can hold as the sequence of matrices in the left-hand side and the function in the right-hand side satisfy the conditions of Definition~\ref{GLT_def}. We want to show that the GLT relation \eqref{GLTotimesGLT=GLT} actually holds.

By Definition~\ref{GLT_def} applied to the GLT sequences $\{A_{n,1}\}_n,\ldots,\{A_{n,d}\}_n$, for every $i=1,\ldots,d$, there exist functions $a_{j,m}^{(i)}$, $f_{j,m}^{(i)}$, $j=1,\ldots,N_m^{(i)}$, such that
\begin{itemize}[nolistsep,leftmargin=*]
	\item $a_{j,m}^{(i)}:[0,1]^{|\nn_i|}\to\mathbb C$ is continuous a.e.\ and $f_{j,m}^{(i)}:[-\pi,\pi]^{|\nn_i|}\to\mathbb C^{s_i\times t_i}$ belongs to $L^1([-\pi,\pi]^{|\nn_i|})$,
	\item $\kappa_m^{(i)}(\xx_i,\btheta_i)=\sum_{j=1}^{N_m^{(i)}}a_{j,m}^{(i)}(\xx_i)f_{j,m}^{(i)}(\btheta_i)\to\kappa_i(\xx_i,\btheta_i)$ a.e.\ on $[0,1]^{|\nn_i|}\times[-\pi,\pi]^{|\nn_i|}$,
	\item $\{B_{n,m}^{(i)}\}_n=\{\sum_{j=1}^{N_m^{(i)}}D_{\nn_i}(a_{j,m}^{(i)}I_{s_i})T_{\nn_i}(f_{j,m}^{(i)})\}\xrightarrow{\rm a.c.s.}\{A_{n,i}\}_n$.
\end{itemize}
Consider the matrix $B_{n,m}^{(1)}\otimes\cdots\otimes B_{n,m}^{(d)}$, which can be written as follows:
\begin{align*}
&B_{n,m}^{(1)}\otimes\cdots\otimes B_{n,m}^{(d)}=\sum_{j_1=1}^{N_m^{(1)}}\cdots\sum_{j_d=1}^{N_m^{(d)}}D_{\nn_1}(a_{j_1,m}^{(1)}I_{s_1})T_{\nn_1}(f_{j_1,m}^{(1)})\otimes\cdots\otimes D_{\nn_d}(a_{j_d,m}^{(d)}I_{s_d})T_{\nn_d}(f_{j_d,m}^{(d)})\quad\mbox{\footnotesize(by {\bf P2})}\\
&=\sum_{j_1=1}^{N_m^{(1)}}\cdots\sum_{j_d=1}^{N_m^{(d)}}(D_{\nn_1}(a_{j_1,m}^{(1)}I_{s_1})\otimes\cdots\otimes D_{\nn_d}(a_{j_d,m}^{(d)}I_{s_d}))(T_{\nn_1}(f_{j_1,m}^{(1)})\otimes\cdots\otimes T_{\nn_d}(f_{j_d,m}^{(d)}))\quad\mbox{\footnotesize(by {\bf P4})}\\
&=\sum_{j_1=1}^{N_m^{(1)}}\cdots\sum_{j_d=1}^{N_m^{(d)}}\Pi_{N(\nn_1),\ldots,N(\nn_d)}^{s_1,\ldots,s_d}D_\nn((a_{j_1,m}^{(1)}\otimes\cdots\otimes a_{j_d,m}^{(d)})I_{s_1\cdots s_d})(\Pi_{N(\nn_1),\ldots,N(\nn_d)}^{s_1,\ldots,s_d})^T\cdot\\
&\qquad{}\cdot\Pi_{N(\nn_1),\ldots,N(\nn_d)}^{s_1,\ldots,s_d}T_\nn(f_{j_1,m}^{(1)}\otimes\cdots\otimes f_{j_d,m}^{(d)})(\Pi_{N(\nn_1),\ldots,N(\nn_d)}^{t_1,\ldots,t_d})^T\quad\mbox{\footnotesize(by Theorems~\ref{thm:TotimesT=T} and~\ref{thm:DotimesD=D})}\\
&=\Pi_{N(\nn_1),\ldots,N(\nn_d)}^{s_1,\ldots,s_d}\left[\sum_{j_1=1}^{N_m^{(1)}}\cdots\sum_{j_d=1}^{N_m^{(d)}}D_\nn((a_{j_1,m}^{(1)}\otimes\cdots\otimes a_{j_d,m}^{(d)})I_{s_1\cdots s_d})T_\nn(f_{j_1,m}^{(1)}\otimes\cdots\otimes f_{j_d,m}^{(d)})\right](\Pi_{N(\nn_1),\ldots,N(\nn_d)}^{t_1,\ldots,t_d})^T.
\end{align*}
Note that the permutation matrices in the middle of the second-to-last expression have disappeared as they are inverses of each other; see also Remark~\ref{cancel_out}.
Let us now make the following observations.
\begin{itemize}[nolistsep,leftmargin=*]
	\item For every $m$, we have
	\begin{align}\label{GLT4.1}
	&\{(\Pi_{N(\nn_1),\ldots,N(\nn_d)}^{s_1,\ldots,s_d})^T(B_{n,m}^{(1)}\otimes\cdots\otimes B_{n,m}^{(d)})\Pi_{N(\nn_1),\ldots,N(\nn_d)}^{t_1,\ldots,t_d}\}_n\notag\\
	&\sim_{\rm GLT}\sum_{j_1=1}^{N_m^{(1)}}\cdots\sum_{j_d=1}^{N_m^{(d)}}(a_{j_1,m}^{(1)}(\xx_1)\cdots a_{j_d,m}^{(d)}(\xx_d))(f_{j_1,m}^{(1)}(\btheta_1)\otimes\cdots\otimes f_{j_d,m}^{(d)}(\btheta_d))\quad\mbox{\footnotesize(by {\bf GLT3})}\notag\\
	&=\sum_{j_1=1}^{N_m^{(1)}}\cdots\sum_{j_d=1}^{N_m^{(d)}}(a_{j_1,m}^{(1)}(\xx_1)f_{j_1,m}^{(1)}(\btheta_1)\otimes\cdots\otimes a_{j_d,m}^{(d)}(\xx_d)f_{j_d,m}^{(d)}(\btheta_d))\quad\mbox{\footnotesize(by {\bf P2})}\notag\\
&=\left(\sum_{j_1=1}^{N_m^{(1)}}a_{j_1,m}^{(1)}(\xx_1)f_{j_1,m}^{(1)}(\btheta_1)\right)\otimes\cdots\otimes\left(\sum_{j_d=1}^{N_m^{(d)}}a_{j_d,m}^{(d)}(\xx_d)f_{j_d,m}^{(d)}(\btheta_d)\right)\quad\mbox{\footnotesize(by {\bf P4})}\notag\\
	&=\kappa_m^{(1)}(\xx_1,\btheta_1)\otimes\cdots\otimes\kappa_m^{(d)}(\xx_d,\btheta_d)=\kappa_m(\xx,\btheta),
	\end{align}
	where the last equality is just the definition of the function $\kappa_m(\xx,\btheta)$.
	\item We have
	\begin{equation}\label{GLT4.2}
	\kappa_m(\xx,\btheta)\to\kappa(\xx,\btheta)\ \,\mbox{a.e.\ on}\ \,[0,1]^{|\nn|}\times[-\pi,\pi]^{|\nn|}
	\end{equation}
	because $\kappa_m^{(i)}(\xx_i,\btheta_i)\to\kappa_i(\xx_i,\btheta_i)$ a.e.\ on $[0,1]^{|\nn_i|}\times[-\pi,\pi]^{|\nn_i|}$ for every $i=1,\ldots,d$.
	\item Since the GLT sequences $\{A_{n,1}\}_n,\ldots,\{A_{n,d}\}_n$ are s.u.\ by Remark~\ref{GLT->s.u.}, we have
	\[ \{B_{n,m}^{(1)}\otimes\cdots\otimes B_{n,m}^{(d)}\}_n\xrightarrow{\rm a.c.s.}\{A_{n,1}\otimes\cdots\otimes A_{n,d}\}_n \]
	by Corollary~\ref{a.c.s.otimes-d}. Thus, we also have
	\begin{align}\label{GLT4.3}
	&\{(\Pi_{N(\nn_1),\ldots,N(\nn_d)}^{s_1,\ldots,s_d})^T(B_{n,m}^{(1)}\otimes\cdots\otimes B_{n,m}^{(d)})\Pi_{N(\nn_1),\ldots,N(\nn_d)}^{t_1,\ldots,t_d}\}_n\notag\\
	&\xrightarrow{\rm a.c.s.}\{\{(\Pi_{N(\nn_1),\ldots,N(\nn_d)}^{s_1,\ldots,s_d})^T(A_{n,1}\otimes\cdots\otimes A_{n,d})\Pi_{N(\nn_1),\ldots,N(\nn_d)}^{t_1,\ldots,t_d}\}_n
	\end{align}
	by definition of a.c.s.\ (Definition~\ref{a.c.s.}), because permutation matrices do not alter the rank and the norm of matrices.
\end{itemize}
Using {\bf GLT4} in combination with \eqref{GLT4.1}--\eqref{GLT4.3}, we obtain \eqref{GLTotimesGLT=GLT}.
\end{proof}

\section{Applications}\label{sec:a}


In this section, we discuss some applications of the main result (Theorem~\ref{thm:GLTotimesGLT=GLT}).

We first observe that the ``purely Toeplitz version'' of Theorem~\ref{thm:GLTotimesGLT=GLT}, i.e., Theorem~\ref{thm:TotimesT=T}, has already proved to be useful in applications.
For example, it was used in \cite{bgd} for developing the theory of multilevel block GLT sequences and in \cite{simax2015} for the analysis of stiffness matrices arising from Lagrangian finite element discretizations of elliptic problems. We remark that both \cite{bgd,simax2015} made use of simplified versions of Theorem~\ref{thm:TotimesT=T} and did not provide an explicit construction of the involved permutation matrices. In other words, Theorem~\ref{thm:TotimesT=T} is a generalization of the simplified results used in \cite{bgd,simax2015} and, moreover, it provides a clear definition of the involved permutation matrices.

Just like Theorem~\ref{thm:TotimesT=T}, even Theorem~\ref{thm:GLTotimesGLT=GLT} has already been used in recent publications~\cite{rositaA,rositaB}.
In particular, it was used in \cite{rositaA} for the 
analysis of preconditioned Toeplitz matrices and in \cite{rositaB} for the 
analysis of two-by-two block linear systems arising from the discretization of space-time fractional diffusion equations.
Both \cite{rositaA,rositaB} leveraged on a simplified version of Theorem~\ref{thm:GLTotimesGLT=GLT}, which was presented therein without a proof.
In this regard, Theorem~\ref{thm:GLTotimesGLT=GLT} provides the first mathematical proof of (a generalized version of) a result, which has already been used in recent literature.

Besides the aforementioned applications, Theorem~\ref{thm:GLTotimesGLT=GLT} can be used for computing the spectral and singular value distribution of sequences of matrices arising from the discretization of partial differential equations (PDEs) whenever such matrices show a tensor-product structure. This usually happens if a tensor-product discretization is adopted for the PDE under consideration.
For example, consider the $d$-dimensional Poisson problem\,\footnote{\,For all the details of the following derivation, we refer the reader to \cite[Section~7.6]{GLTbookII}.}
\begin{equation}\label{Poisson}
\begin{cases}
-\Delta u=f, &\mbox{on $(0,1)^d$,}\\
u=0, &\mbox{on $\partial(0,1)^d$.}
\end{cases}
\end{equation}
Suppose we discretize \eqref{Poisson} by the Galerkin finite element method based on tensor-product B-splines of degree $\pp=(p_1,\ldots,p_d)\in\mathbb N^d$ over the uniform grid $\ii/\nn=(i_1/n_1,\ldots,i_d/n_d)$, $\ii=\bz,\ldots,\nn$, where $\nn=\nn(n)\in\mathbb N^d$ depends on a mesh fineness parameter $n$ and $\nn\to\infty$ as $n\to\infty$.
In this case, the actual computation of the numerical solution reduces to solving a linear system with coefficient matrix given by
\begin{equation}\label{d_matrix}
A_\nn^{(\pp)}=\sum_{r=1}^dM_{n_1}^{(p_1)}\otimes\cdots\otimes M_{n_{r-1}}^{(p_{r-1})}\otimes K_{n_r}^{(p_r)}\otimes M_{n_{r+1}}^{(p_{r+1})}\otimes\cdots\otimes M_{n_d}^{(p_d)},
\end{equation}
where, for all positive integers $m,p\ge1$, $K_m^{(p)}$ and $M_m^{(p)}$ are real symmetric matrices defined as follows:
\begin{align}
\label{Kmatrix}K_m^{(p)}=\left[\int_0^1B_{j+1,p,m}'(x)B_{i+1,p,m}'(x){\rm d}x\right]_{i,j=1}^{m+p-2},\\
\label{Mmatrix}M_m^{(p)}=\left[\int_0^1B_{j+1,p,m}(x)B_{i+1,p,m}(x){\rm d}x\right]_{i,j=1}^{m+p-2},
\end{align}
with $B_{2,p,m},\ldots,B_{n+p-1,p,m}$ being the B-splines of degree $p$ defined over the uniform grid $i/m$, $i=0,\ldots,m$, and vanishing on the boundary of the interval $(0,1)$.
From the GLT analysis of the matrices $K_m^{(p)}$ and $M_m^{(p)}$, we know that, for every $p\ge1$,
\begin{align}
\label{K-GLT}\{K_m^{(p)}\}_m&\sim_{\rm GLT}f_p,\\
\label{M-GLT}\{M_m^{(p)}\}_m&\sim_{\rm GLT}h_p,
\end{align}
where $f_p,h_p:[-\pi,\pi]\to\mathbb R$ are real trigonometric polynomials that are explicitly known in closed form.
As a consequence of \eqref{K-GLT}--\eqref{M-GLT}, we have, for every $\pp\in\mathbb N^d$,
\begin{alignat}{3}
\label{Ki-GLT}\{K_{n_i}^{(p_i)}\}_n&\sim_{\rm GLT}f_{p_i},&\qquad i=1,\ldots,d,\\
\label{Mi-GLT}\{M_{n_i}^{(p_i)}\}_n&\sim_{\rm GLT}h_{p_i},&\qquad i=1,\ldots,d.
\end{alignat}
We can now apply the main Theorem~\ref{thm:GLTotimesGLT=GLT} in combination with {\bf GLT3} to conclude that, for every $\pp\in\mathbb N^d$,
\begin{equation}\label{GLTeasy}
\{A_\nn^{(\pp)}\}_n\sim_{\rm GLT}\sum_{r=1}^dh_{p_1}\otimes\cdots\otimes h_{p_{r-1}}\otimes f_{p_r}\otimes h_{p_{r+1}}\otimes\cdots\otimes h_{p_d}.
\end{equation}
Since each matrix $A_\nn^{(p)}$ is a sum of tensor products of real symmetric matrices, it is itself a real symmetric matrix by {\bf P3}.
Thus, by \eqref{GLTeasy} and {\bf GLT1},
\begin{equation}\label{GLTeasy'}
\{A_\nn^{(\pp)}\}_n\sim_{\sigma,\lambda}\sum_{r=1}^d h_{p_1}\otimes\cdots\otimes h_{p_{r-1}}\otimes f_{p_r}\otimes h_{p_{r+1}}\otimes\cdots\otimes h_{p_d}. 
\end{equation}
We remark that the GLT relation \eqref{GLTeasy}, which was obtained in \cite[Section~7.6]{GLTbookII} with a certain effort, has now been obtained directly from the main Theorem~\ref{thm:GLTotimesGLT=GLT} and \eqref{Ki-GLT}--\eqref{Mi-GLT}. In other words, Theorem~\ref{thm:GLTotimesGLT=GLT} allows us to reduce the (complicated) GLT analysis of a $d$-dimensional matrix such as $A_\nn^{(\pp)}$ to the (easy) GLT analysis of its ``unidimensional pieces'' $K_m^{(p)}$ and $M_m^{(p)}$.

\section{Conclusions}\label{sec:c}

In our main result (Theorem~\ref{thm:GLTotimesGLT=GLT}), we have proved that, if $\{A_{n,1}\}_n,\ldots,\{A_{n,d}\}_n$ are GLT sequences with symbols $\kappa_1,\ldots,\kappa_d$, their tensor product $\{A_{n,1}\otimes\cdots\otimes A_{n,d}\}_n$ is a GLT sequence with symbol $\kappa_1\otimes\cdots\otimes\kappa_d$, up to suitable permutation matrices that only depend on the dimensions of the involved matrices $A_{n,1},\ldots,A_{n,d}$. Moreover, the permutation matrices in question are explicitly defined by \eqref{Pi-mu} and can be computed through the recursive formula in Definition~\ref{Gamma(sigma)}.
While proving Theorem~\ref{thm:GLTotimesGLT=GLT}, we have also proved other new results that are significant enough to be considered as further main results of this paper in addition to Theorem~\ref{thm:GLTotimesGLT=GLT}, namely Theorems~\ref{a.c.s.otimes}, \ref{thm:TotimesT=T}, \ref{thm:DotimesD=D} and Corollary~\ref{a.c.s.otimes-d}.
After proving Theorem~\ref{thm:GLTotimesGLT=GLT}, we have discussed some applications in Section~\ref{sec:a}.

We conclude this paper by suggesting a possible future line of research. Besides tensor products, another important matrix operation is the direct sum.
If $X\in\mathbb C^{m\times n}$ and $Y\in\mathbb C^{p\times q}$, the direct sum of $X$ and $Y$ is the $(n+p)\times(m+q)$ matrix defined by
\[ X\oplus Y={\rm diag}(X,Y)=\begin{bmatrix}X & O\\ O & Y\end{bmatrix}. \]
Let $\{A_{n,1}\}_n,\ldots,\{A_{n,d}\}_n$ be GLT sequences with symbols $\kappa_1,\ldots,\kappa_d$, and consider their direct sum $\{A_{n,1}\oplus\cdots\oplus A_{n,d}\}_n$.
Can we say that $\{A_{n,1}\oplus\cdots\oplus A_{n,d}\}_n$ is a GLT sequence? If yes, which is the relation between its symbol and the symbols $\kappa_1,\ldots,\kappa_d$?
Answering to these questions would be useful in the analysis of block matrices formed by ``GLT blocks'' such as those considered in \cite{blocking}. In fact, a partial answer in the case where $A_{n,1},\ldots,A_{n,d}$ are pure ($1$-level block) Toeplitz matrices was provided in \cite[Theorem~3.12]{blocking}.
Providing a definitive answer in the case of arbitrary GLT sequences is still an open problem, whose solution may form the content of a future research.


\section*{Acknowledgements}

{\footnotesize

The author is member of the research group GNCS (Gruppo Nazionale per il Calcolo Scientifico) of INdAM (Istituto Nazionale di Alta Matematica).
This work was supported by the Department of Mathematics of the University of Rome Tor Vergata through the projects MatMod@TOV (MUR excellence department project, CUP E83C23000330006) and METRO (Methods and modEls for arTificial neuRal netwOrks, CUP E83C25000630005).}

\end{document}